\documentclass[EJP,preprint]{ejpecp} 

\usepackage{enumerate}  

\SHORTTITLE{Return probability on Galton--Watson trees}

\TITLE{On the return probability of the simple random walk on Galton--Watson trees\support{This work 
	was partially supported by the Deutsche Forschungsgemeinschaft 
	(DFG, German Research Foundation) -- TRR 352 ``Mathematics of Many-Body Quantum Systems and 
	Their Collective Phenomena" -- Project-ID 470903074.}}

\AUTHORS{%
  Peter~M\"uller\footnote{Mathematisches Institut, Ludwig-Maximilians-Universit\"at M\"unchen,
  	Theresienstra\ss{e} 39, 80333 M\"unchen, Germany. \BEMAIL{mueller@lmu.de}}\orcid{0000-0002-3063-9636}
  \and 
  Jakob~Stern\footnotemark[2]}


\AMSSUBJ{05C81; 
	60K37; 
	60J80} 

\SUBMITTED{February 27, 2024} 
\ACCEPTED{January 2, 2025} 

\VOLUME{0}
\YEAR{2023}
\PAPERNUM{0}
\DOI{10.1214/YY-TN}

\ABSTRACT{We consider the simple random walk on Galton--Watson trees with supercritical offspring 
	distribution, 
	conditioned on non-extinction. In case the offspring distribution has 
	finite support, we prove an upper bound for the annealed return probability to the root which decays 
	subexponentially in time with exponent $1/3$. This exponent is optimal. Our result improves 
	the previously known subexponential upper bound with exponent $1/5$ by Piau 
	[Ann.\ Probab.\ \textbf{26}, 1016--1040 (1998)]. For offspring distributions with unbounded support
	but sufficiently fast decay, our method also yields improved subexponential upper bounds. 
}

\renewcommand{\le}{\leqslant}
\renewcommand{\ge}{\geqslant}
\renewcommand{\leq}{\leqslant}
\renewcommand{\geq}{\geqslant}
\let\dotlessi\i  \renewcommand{\i}{\mathrm{i}}
\renewcommand{\d}{\mathrm{d}}  

\let\emptyset\myemptyset

\DeclareMathOperator*{\essinf}{ess\,inf}
\DeclareMathOperator{\dist}{dist}
\providecommand{\LEFTRIGHT}[3]{\left#1 #3 \right#2}

\providecommand{\wtilde}[1]{\widetilde{#1}}
\providecommand{\clap}{\text}
\newcommand{\Gbf}{\mathbf{G}}
\newcommand{\Tbf}{\mathbf{T}}

\newcommand{\Nbb}{\mathbb{N}}
\newcommand{\Rbb}{\mathbb{R}}
\newcommand{\Tbb}{\mathbb{T}}

\renewcommand{\phi}{\varphi}
\renewcommand{\theta}{\vartheta}
\renewcommand{\subsetneq}{\varsubsetneqq}
\newcommand{\vv}[1]{|#1|}
\newcommand{\ev}[1]{|#1|}

\begin{document}

\section{Introduction and result}
Galton--Watson trees are extensively studied as one of the classic examples of trees \cite{peres2016prob}. They arise naturally in many contexts, for example, as the local weak limit of sparse Erd\H{o}s--R\'enyi random graphs \cite{virag2017spectrans,hofstad2020glimit}, i.e.\ those for which the mean number of edges grows proportionally to the number of vertices. 

To fix notation, let $\Nbb \coloneqq \{1,2,\ldots\}$ denote the set of positive integers and 
$\Nbb_{0} \coloneqq \Nbb \cup \{0\}$.
We write $\{p_j\}_{j\in\Nbb_0} \in [0,1[^{\,\Nbb_{0}}$, with $\sum_{j\in\Nbb_{0}}p_j=1$,
for the (non-degenerate) offspring distribution of a Galton--Watson branching process with a single progenitor. 
The associated family tree is called a Galton--Watson tree. We adhere to its canonical realisation as a probability space $(\Tbb,\mathcal{F}, G^{*})$ of rooted tree graphs $\Tbf\in\Tbb$ where the probability measure $G^{*}$ is canonically determined by the 
offspring distribution $\{p_j\}_{j\in\Nbb_0}$. 
Throughout we assume that the Galton--Watson tree is supercritical, that is, with mean number of offspring
\begin{equation}
	\label{lambda-def}
	\lambda  \coloneqq \sum_{j\in\Nbb_{0}} j p_{j} > 1.
\end{equation}
As $p_{1} \neq 1$, this is equivalent \cite[Prop.\ 5.4]{peres2016prob} 
to the Galton--Watson tree having a chance not to die out, i.e.\
\begin{equation} \label{g-infty-def}
	g_{\infty}\coloneqq G^{*}\big[\{ T\in\Tbb: \vv{\Tbf} = \infty\}\big] >0.
\end{equation}
Here, $|\pmb\cdot|$ indicates the cardinality of a set, and whenever there is no danger of confusion, our notation will not distinguish between a graph and its vertex set. Thus, $\vv{\Tbf}$ denotes the number of vertices of the tree $\Tbf$. 
We write 
\begin{equation}
	G  \coloneqq  G^{*} \big[\;\pmb\cdot \, \big|\; \{T\in\Tbb: \vv{\Tbf} = \infty\}\big]
\end{equation}
for the conditional probability measure conditioned on non-extinction.  
	
Random walks in random environments have been studied in numerous works. 
We refer to \cite[Chap.\ 16]{peres2016prob} and references therein for a nice account of random walks on 
a Galton--Watson tree. Despite extensive research on this topic, there exists no comprehensive, 
sharp result for the annealed return probability to the root of the simple random walk 
on a Galton--Watson tree. This is not only a natural question, but also links to other 
interesting quantities like spectral properties of the random walk's generator, 
the graph Laplacian on a Galton--Watson tree.

Let $P[\,\pmb\cdot\,] \coloneqq P_{o}^{\Tbf}[\,\pmb\cdot\,]$ denote the probability measure of 
the \emph{simple} 
random walk $\{X_t\}_{t\in\Nbb_0}$ on a realisation $\Tbf\in\Tbb$ of a Galton--Watson tree, 
starting at the root $o$. 
In particular, $\{X_t\}_{t\in\Nbb_0}$ jumps only along edges of $\Tbf$, and if $x$ and $y$ are the two vertices of an edge in $\Tbf$ then the transition probability for $\{X_t\}_{t\in\Nbb_0}$ to jump from $x$ to $y$ is given by the inverse of the vertex degree of $x$.
We then write
\begin{equation}
	\label{eq:rt-def}
	R_t \coloneqq  GP[X_{2t}=o]  \coloneqq  \int_{\Tbb}\d G(\Tbf) P_{o}^{\Tbf}[X_{2t}=o]
\end{equation}
for the annealed return probability to the root at time $t\in\Nbb_0$ of the simple random walk and summarise the known results in Table~\ref{tab:piau}.
%
\begin{table}
\begin{center}
\begin{tabular}{@{}l@{\qquad} p{4.4cm} @{\;}r @{\;\,} c @{\;\,} l@{\;\;}}
	(a) & $p_0=p_1=0$: & $\exp(-c't)\leq$ & $ R_t $ & $\leq \exp(-ct)$\\	
	(b) &$p_0>0 \text{~~or~~} p_1>0$: & $\exp(-c't^{\frac13})\leq$ & $R_t$ & \\	
	(c) &$p_0=0$: &  & $R_t$ & $\leq \exp(-ct^{\frac13})$\\	
	(d) &$\{p_j\}_{j\in\Nbb_0}$ finitely supported: &  & $R_t$ & $\leq \exp(-ct^{\frac15})$\\	
	(e) & general $\{p_j\}_{j\in\Nbb_0}$: &  & $R_t$ & $\leq \exp(-ct^{\frac16})$.\\
\end{tabular} 
\end{center}
\caption{Summary of known bounds \cite{piau1998lowbnd} for the annealed return probability \protect{\eqref{eq:rt-def}} to the root. Here, $c,c'>0$ are constants (independent of $t$), which may differ from line to line, and the bounds hold for all $t\in\Nbb_0$.}
\label{tab:piau}
\end{table}
%

\begin{remark}
	\begin{enumerate}[(i)]
	\item 
		The statements in Table~\ref{tab:piau} follow from \cite[Thm.\ 2]{piau1998lowbnd}, which establishes  
		corresponding results for the tail of the annealed distribution $GP[\tau_{R} \ge t]$ of the first regeneration time 
		$\tau_{R}:= \min\{t \in\Nbb: X_{t_{-}} \neq X_{t} \;\forall\, t_{-} <t \text{ and }X_{t-1} \neq X_{t_{+}} \;\forall\, t_{+} \ge t\}$. More precisely, concerning the upper bounds, this is a direct consequence of 
		the inclusion
		$\{X_{t}=o\} \subset \{\tau_{R} \ge t\}$. The lower bounds on $R_{t}$ follow from analogous ideas as used for the lower
		bounds on $GP[\tau_{R} \ge t]$ in \cite{piau1998lowbnd}.
	\item \label{noConstant}
		We refrained from introducing multiplicative constants in front of the exponentials or 
		minimal times for which the results hold because they can be absorbed in the constants 
		$c$ and $c'$ in the exponent. This is always possible because $R_{t} <1$ for every $t\in\Nbb$, 
		which follows, e.g., from the random walk having positive speed \cite[p.\ 569]{peres2016prob} or 
		being transient \cite{Collevecchio.2006}, see also \cite[Lemma~2]{Grimmett.2001} which was 
		announced in \cite{GRIMMETT.1984}. 
	\item 
		Case (a) in Table~\ref{tab:piau} differs from the other cases and is quite well understood: 
		The exponential decay in time results from the random	walk getting lost in a tree where the 
		number of vertices at least doubles in each generation and where there are no deterministic 
		push-backs due to the absence of leaves. Using the inclusion of events 
		\begin{equation}
			\label{large-dev-rel}
 			\{X_{2t}=o\}\subseteq \Big\{\frac{\mathrm{dist}\{o, X_{2t}\}}{2t}\in [0,\epsilon[\,\Big\}
		\end{equation}
		for any $\epsilon>0$, the bound $c'\geq\ln\frac{9}{8}$ for the constant in the lower bound of (a) 
		follows from large-deviation estimates of the speed 
		in \cite[Thm.\ 1.2]{nina2001speed}. In \eqref{large-dev-rel}, 
		$\mathrm{dist}\{\,\pmb\cdot\,,\,\pmb\cdot\,\}$ stands for the graph distance.
	\item	
		The parameter range in (b) of Table~\ref{tab:piau} is complementary to that in (a). 
		The subexponential behaviour with 
		exponent $\frac13$ in the lower bound of (b) is believed to capture the exact long-time 
		asymptotics of $R_{t}$ in this parameter regime. Unfortunately, corresponding upper bounds are not 
		known in such generality but only in the absence of leaves as specified in (c). In this special case,
		a coupling argument allows one to compare $\{X_t\}_{t\in\Nbb_0}$ with a simple random walk on 
		the non-negative integers from which the result follows. The best upper bound valid for all offspring 
		distributions allowed in (b)---and also in (a)---exhibits only a subexponential decay with exponent 
		$\frac16$. If the offspring distribution limits the number of descendants to some maximum as in the 
		case (d), then the exponent of the subexponential upper bound improves to $\frac15$.
	\end{enumerate}
 	\label{PiauRem}
\end{remark}
	
In this paper, we prove a subexponential upper bound for $R_{t}$ with the optimal exponent $\frac13$ 
in the situation of (d) in Table~\ref{tab:piau}. Our method of proof builds upon arguments of Vir\'ag in 
\cite{virag2000bnddegree}. Vir\'ag considers a deterministic, i.e.\ a fixed,  weighted graph 
$\mathbf G$ with a maximal vertex degree $w_{0}$  
and obtains the existence of a minimal time $t({\mathbf G}) \in\Nbb$ such that for all times 
$t \ge t(\mathbf G)$ and all vertices $u,v$ of $\mathbf G$ the heat kernel of the simple random walk on 
$\mathbf G $ obeys the bound \cite[Thm.\ 1.3]{virag2000bnddegree}
\begin{equation}
	\label{vir-est}
	P_{u}^{\mathbf G}[X_{t} = v] \le \exp\{-\alpha(\mathbf G)\, t^{\frac13}\} 
\end{equation}
with an explicit constant $\alpha(\mathbf G) \coloneqq \big(\mkern3mu\boldsymbol{\wtilde{\textbf{\dotlessi}}}
	(\mathbf G)\big)^{2} (w_{0}^{2}/2)^{-\frac13}/9$ given in terms of the \emph{weighted} edge anchored 
expansion constant $\boldsymbol{\wtilde{\textbf{\dotlessi}}}(\mathbf G)$ of $\mathbf G$, see e.g.\
\cite[p.~1589]{virag2000bnddegree} for its definition (denoted there without tilde). Since $\mathbf G$ 
has a maximum vertex degree, the bound \eqref{vir-est} can be reexpressed in terms of the 
\emph{unweighted} edge anchored expansion constant $\mathbf{i}(\Gbf)$ defined in \eqref{T-exp-const}. 
Chen and Peres \cite[Cor.\ 1.3]{ChenPeres2004} proved $\mathbf{i}(\Tbf)>0$ for $G$-a.e.\ 
$\mathbf T \in\Tbb$. Thus, the estimate \eqref{vir-est} gives rise to a \emph{quenched} upper bound for 
the return probability on a Galton--Watson tree for large times, as has also been noted in 
\cite[Ex.\ 6.2]{virag2000bnddegree}. However, this quenched upper bound does not imply an 
\emph{annealed} upper bound for the return probability due to the uncontrolled dependence of the 
minimal time $t(\mathbf G)$ on the graph $\mathbf G$. 

We elaborate on the strategy of \cite{virag2000bnddegree} but replace deterministic arguments which 
lead to the uncontrolled dependence of $t(\mathbf T)$ on $\mathbf T \in\Tbb$ by probabilistic ones. 
This is facilitated by the observation that 
$G\text{-}\essinf_{\Tbf\in\Tbb} \mathbf{i}(\mathbf{T}) > 0$, see Lemma~\ref{hdef}. Thereby we obtain 
Theorem~\ref{Decay} which states an upper bound for the annealed return probability with optimal 
exponent $\frac13$ for offspring distributions with bounded support.

\begin{theorem}\label{Decay}
	Let $\Tbb$ be a supercritical Galton--Watson tree conditioned on non-extinction and assume that its 
	offspring distribution has bounded support. Then, there exists a constant $c>0$ such that for all
	$t\in\Nbb_{0}$ we have 
	\begin{equation}
		\label{eq:virag-cor}
		R_t \leq \exp(- c t^{\frac13}).
	\end{equation}
\end{theorem}

The proof of Theorem~\ref{Decay} is deferred to the end of Section~\ref{sec:proofs}.	
It relies on Theorem~\ref{exactdecay} and an approximation argument.

 In order to treat also offspring
distributions with unbounded support we introduce a time-dependent regularisation of the Galton--Watson 
trees which makes their number of offspring (essentially) bounded. Since the random walk has to return to the root 
within a given time $t$ it does not explore generations above level $t$ (in fact, $t/2$), and thus a 
modification of the tree above this level does not change the return probability. Modifications below 
level $t$ do affect the return probability but can be controlled by the decay of the offspring 
distribution. This comes at the expense of weakening the decay exponent of the upper bound.

\begin{theorem}\label{thm}
	Let $\Tbb$ be a supercritical Galton--Watson tree conditioned on non-extinction and assume that its 
	offspring distribution has a super-Gaussian decay according to 
	\begin{equation}
		\label{eq:k-decay}
		p_j\leq c_1\exp(-c_2j^k) \qquad \text{for every~} j\in\Nbb_0,
	\end{equation}
	where $c_1,c_2>0$ and $k>2$ are constants (all independent of $j$).
	Then, the annealed return probability at time $t$ decays at least exponentially in $t^{\frac{1}{3}-\frac{2}{3k}}$, i.e., there is a constant $c>0$ such that for all $t\in\Nbb_0$ we have 
	\begin{equation}
		R_t\leq \exp(-c t^{\frac{1}{3}-\frac{2}{3k}}).
	\end{equation}
\end{theorem}

The proof of Theorem \ref{thm} is also deferred to the end of Section~\ref{sec:proofs}.

\begin{remark}
		Whereas Theorem~\ref{Decay} improves the exponent $\frac15$ in (d) of Table~\ref{tab:piau} to the optimal value $\frac13$, 
		Theorem~\ref{thm} yields an improvement to the exponent $\frac16$ in (e) of Table~\ref{tab:piau} for $k >4$. 
\end{remark}

For even faster decaying offspring distributions, we get arbitrarily close to the optimal exponent.   
	
\begin{corollary}\label{twiceExp}
	Let $\Tbb$ be a supercritical Galton--Watson tree conditioned on non-extinction and assume that its 
	offspring distribution is very fast decaying according to 
	\begin{equation}
		p_j\leq \exp\big(- \xi(j)\big) \qquad \text{for every~} j\in\Nbb_0 ,
	\end{equation}
	where $\xi: \mathbb{N}_{0} \rightarrow \,]0,\infty[\,$ (independent of $j$) grows faster than any polynomial.  
	Then, for every $\varepsilon>0$, the annealed return probability at time $t$ decays at least
	exponentially in $t^{\frac{1}{3}-\epsilon}$, i.e., there is a constant $c>0$ such that for all 
	$t\in\Nbb_0$ we have 
	\begin{equation}
		R_t\leq \exp(-c t^{\frac{1}{3}-\epsilon}).
	\end{equation}
\end{corollary}

\begin{proof}
	Fix $\epsilon>0$. Then, there is $k>2$ such that $\epsilon\geq\frac{2}{3k}$. 
	Now, we choose $c_1,c_2>0$ such that $p_j\leq c_1\exp(-c_2j^k)$ for every $j\in\Nbb_0$. This is 
	possible, since $\xi$ grows faster than any polynomial. Then, we can apply Theorem~\ref{thm} and 
	obtain for all $t\in\Nbb_0$
	\begin{equation}
		R_t\leq \exp(-ct^{\frac{1}{3}-\frac{2}{3k}})\leq \exp(-ct^{\frac{1}{3}-\epsilon}),
	\end{equation}
	with a constant $c>0$.
\end{proof}

In the rest of this paper, we will conduct the proof of Theorems~\ref{Decay} and \ref{thm}. 
In view of the known bound (c) in Table~\ref{tab:piau}, it suffices to do this 
under the following

\begin{assumption}
	\label{gen-ass}
	Let $\Tbb$ be a supercritical Galton--Watson tree conditioned on non-extinction with 
	offspring distribution satisfying $p_{0}>0$ and decaying as in \eqref{eq:k-decay} for some fixed
	exponent $k>2$.
\end{assumption}

Assumption~\ref{gen-ass} applies to the rest of this paper without mentioning it explicitly 
in our statements below. Boundedness of the support of $\{p_j\}_{j\in\Nbb_0}$, which is required for 
Theorem~\ref{Decay}, will be indicated where it is needed.


\section{Anchored expansion and events of bad trees}

We introduce time-dependent events of bad trees which we will exclude in our further analysis. 
Their negligence will result only in an exponentially small error term in the annealed average. 
In order to formulate these events, we introduce some notation. 

If $\Gbf$ is a graph and $x$ is a vertex of $\Gbf$, we briefly write $x\in\Gbf$ 
(instead of referring to the vertex set of $\mathbf G$). For $\Tbf \in \Tbb$ and a vertex 
$x\in\mathbf{T}$, the number of children of $x$ is denoted by $Z_{\Tbf}(x)$. 
From the branching process point of view, these are the values of i.i.d.-copies of a random variable 
$Z$ with distribution $G^{*}[Z=j] = p_{j}$ for every $j\in\Nbb_{0}$.
The position $X_{t}$ of the random walk on a tree $\Tbf$ with starting point at the root explores at 
most the subtree up to generation $t$ of the tree. We denote this subtree by $\mathbf{T}_{ot}$.
The first bad event consists of subtrees $\mathbf{T}_{ot}$ that possess a vertex with too many 
descendants.

\begin{lemma}\label{SetsOfTrees}
	For $t\in\Nbb$ we define the event 
	\begin{equation}
		\label{Ftdefeq}
		F_{t}  \coloneqq  \big\{\Tbf\in\Tbb: \exists\, x\in \mathbf{T}_{ot}:\; Z_{\Tbf}(x) 
			\ge c_{3} t^{\frac{1}{k}}\big\},
	\end{equation}
	where $c_{3} >3+ (\ln\lambda)/c_{2}$ is a constant. The constants $k>2$ and $c_{2}>0$ were specified in 
	\eqref{eq:k-decay} and
	$\lambda>1$ in \eqref{lambda-def}. Then there exists a constant $C>0$ such that for every $t\in\Nbb$ 
	we have
	\begin{equation}
		G[F_{t}]\leq C\exp(-c_{4}t)
	\end{equation}
	with decay rate $c_{4} \coloneqq \min\{c_{2},c_{2}(c_{3}-3) -\ln\lambda\} >0$.
\end{lemma}

\begin{proof}
	Let $t\in\Nbb$ be fixed.
	We define 
	\begin{equation}
			B_{t}  \coloneqq \big\{\Tbf\in\Tbb: \vv{\mathbf{T}_{ot}}\geq \exp(c't)\big\}
	\end{equation}
	with $c' \coloneqq c_{2}(c_{3}-3)> \ln\lambda$ and decompose 
	\begin{equation}
			F_{t} \subseteq (F_{t} \setminus B_{t}) \cup B_{t}.
	\end{equation}
	As to the decay of the probability of $B_{t}$, we recall the expectation of the size
	\begin{equation}
		\label{eq:Atproof}
		\int_{\Tbb}\!\d G^*(\Tbf) \, \vv{\mathbf{T}_{ot}}
		= \sum_{j=0}^{t}\lambda^j
		= \frac{\lambda^{t+1}-1}{\lambda-1},
	\end{equation}
	from \cite[Prop.\ 5.5]{peres2016prob} and obtain from Chebyshev's inequality 
	\begin{equation}
		\label{eq:At-decay}
		G[B_{t}]\leq\frac{1}{g_{\infty}} \, G^{*}[B_{t}]
		\leq \frac{1}{g_{\infty}} \, \frac{\lambda^{t+1}-1}{\lambda-1}\exp(-c't).
	\end{equation}
	Here, the probability $g_{\infty}>0$ of non-extinction was defined in \eqref{g-infty-def}.
	Due to $c' > \ln\lambda$, the right-hand side of \eqref{eq:At-decay} decays exponentially in $t$ 
	with rate $c' - \ln\lambda >0$.

	Next, we turn to the probability of $F_{t} \setminus B_{t}$. Since we are in the complement of the
	event $B_{t}$ we estimate
	\begin{equation}
		G[F_{t} \setminus B_{t}]\leq\frac{1}{g_{\infty}} \, \exp(c't)
		\sum_{j=\lfloor {c_{3}t^{\frac{1}{k}}} \rfloor}^{\infty}G^{*}[Z=j],  
	\end{equation}
	where $\lfloor r \rfloor$ denotes the largest integer not exceeding $r \in\Rbb$.
	Inserting the decay \eqref{eq:k-decay} of the offspring distribution $p_{j} = G^{*}[Z=j]$ and 
	estimating the resulting sum by an integral, we obtain 
	\begin{equation}
		\label{eq:Bt-decay}
		G[F_{t} \setminus B_{t}] \leq\frac{c_1}{c_{2}\,g_{\infty}}\, \exp(c't)
			\exp\big(-c_2(c_{3}-2)^{k}t\big) 
		\leq  \frac{c_1}{c_{2}\,g_{\infty}}\,\exp(-c_2t).
	\end{equation}
	Combining \eqref{eq:At-decay} and \eqref{eq:Bt-decay}, the claim follows.
\end{proof}

\noindent
The following version of an (edge) anchored expansion constant of a Galton--Watson tree $\Tbb$ 
will be useful to us.	
\begin{definition}
	Given a rooted tree $\Tbf$ for which each vertex has finite degree, we set 
		\begin{equation}
		\label{T-exp-const}
 		\mathbf{i}(\mathbf{T})  \coloneqq  \lim_{n\to\infty}\inf\Bigl\{\frac{\ev{\partial K}}{\vv{K}}:\;o\in K\subset\mathbf{T}\mbox{ connected}, n\leq \vv{K}<\infty\Bigr\}.
	\end{equation}
	Here $\partial K$ denotes the edge boundary of $K$, i.e.\ the set of edges 
	of $\Tbf$ which connect a vertex in the subgraph $K$ with a vertex in the complement $\Tbf\setminus K$, and $\ev{\partial K}$ is its cardinality. 
	The anchored expansion constant of a Galton--Watson tree is then defined as
	\begin{equation}
		\label{TTbounds}
 		\mathbf{i}_{\mkern2mu\Tbb}  \coloneqq  G\mkern-1mu\text{-}\mkern-1mu\essinf_{\hspace{-1em}\Tbf\in\Tbb} 
			\,\mathbf{i}(\mathbf{T}).
	\end{equation}
\end{definition}

\noindent
Chen and Peres \cite[Cor.\ 1.3]{ChenPeres2004}, see also  \cite[Thm.\ 6.52]{peres2016prob}, 
proved $\mathbf{i}(\mathbf{T}) >0$ for $G$-a.e.\ $\Tbf\in\Tbb$ without any further assumptions on 
the offspring distribution besides being supercritical. 
This can be strengthened. 

\begin{lemma}\label{hdef}
	Assume that $\Tbb$ is supercritical. Then 
	\begin{equation}\label{iTTbounds}
 		\mathbf{i}_{\mkern2mu\Tbb} > 0.
	\end{equation}
\end{lemma}

\begin{remark}
	Assumption~\ref{gen-ass} is not used for Lemma~\ref{hdef}.
\end{remark}

\begin{proof}[Proof of Lemma~\ref{hdef}]
	We follow and build upon the proof of \cite[Thm.\ 6.52]{peres2016prob}.  \\
	\emph{Case {\upshape (i):} $p_{0}=p_{1}=0$}.~~ With each vertex having at least two descendants, the estimate 
	\begin{equation}
		|K| \le |\partial K| \sum_{j\in\Nbb} 2^{-j} = |\partial K|
	\end{equation}
	for any $K$ as in \eqref{T-exp-const} yields $\mathbf{i}_{\mkern2mu\Tbb} \ge 1$. \\
	\emph{Case {\upshape (ii):} $p_{0}=0 \neq p_{1}$}.~~ As in (ii) of the proof of \cite[Thm.\ 6.52]{peres2016prob}, this case is obtained from a random subdivision of Case (i). We write $\mathbf{i}_{\mkern2mu\Tbb, (\i)}$ 
	for the anchored expansion constant from Case (i), whence $\mathbf{i}_{\mkern2mu\Tbb, (\i)} \ge 1$. The proof of \cite[Thm.\ 6.50]{peres2016prob} shows that there exists a constant $c >0$, which is determined only by the rate function of large deviations for the geometric distribution, such that 
	\begin{equation}
		\mathbf{i}_{\mkern2mu\Tbb} \ge \frac{1}{c}\, \frac{1}{1+ 1/\mathbf{i}_{\mkern2mu\Tbb, (\i)}}
		\ge \frac{1}{2c}.
	\end{equation}
	 \\
	\emph{Case {\upshape (iii):} $p_{0} \neq 0$}.~~
	It is shown in (iii) of the proof of \cite[Thm.\ 6.52]{peres2016prob} that, 
	given any $h>0$ sufficiently small, the probability of the events 
	\begin{equation}\label{Adef}
 		A(h,n) \coloneqq \Bigl\{\mathbf{T}\in\Tbb: \exists\, K \subset \mathbf{T} \mbox{ connected with}\; o\in K, |K|= n, |\partial K| \leq hn \Bigr\}
	\end{equation}
	decays exponentially
	\begin{equation}\label{Abounds}
 		G^*[A(h,n)]\leq \exp(-c_hn)
	\end{equation}
	for $n\in\Nbb$, where $c_h>0$ is a constant depending on $h$ (but not on $n$).
	Hence, we have $\sum_{n\in\Nbb}G[A(h,n)]\leq\frac{1}{g_{\infty}}\sum_{n\in\Nbb}G^*[A(h,n)]<\infty$, 
	and the Borel--Cantelli lemma implies that the event $A(h) \coloneqq \limsup_{n\to\infty}A(h,n)$ is a $
	G$-null set. We conclude that 
	\begin{equation}
 		\mathbf{i}_{\mkern2mu\Tbb} \geq \inf_{\Tbf\in\Tbb\setminus A(h)} \; \lim_{n\to\infty}
			\inf\Bigl\{\frac{|\partial K|}{|K|}:\;o\in K\subset\mathbf{T}\in\Tbb\mbox{ connected}, 
				n\leq |K|<\infty\Bigr\} 
		\geq h>0.  
	\end{equation}
\end{proof}

\noindent
The estimate \eqref{Abounds} immediately implies the next lemma. It bounds the probability of the
event $\bigcup_{n\ge t} A(h,n)$, which we will also exclude later. 
Since it is important to us, we state it separately.

\begin{lemma} \label{lem:hdef}
	There exist constants $h \in \,]0, \mathbf{i}_{\mkern2mu\Tbb}[\,$ and $c_{5}>0$ such that for every 
	$t\in\Nbb$ the event 
	\begin{equation}
		\label{D_t-def}
		D_t \coloneqq \bigcup_{n\ge t} A(h,n) = \Bigl\{\Tbf \in \Tbb: \exists\, o\in K\subset\Tbf\mbox{ connected}, t\leq \vv{K}<\infty, 
			\frac{\ev{\partial K}}{\vv{K}}\leq h\Bigr\}
	\end{equation}
	has exponentially small probability  
	\begin{equation} 
		\label{D_t-bounds}
		G[D_t]\leq\exp(-c_{5}t).
	\end{equation}
\end{lemma}

\begin{remark}
	\label{rem:h-small}
	Since $A(h',n) \subseteq A(h,n)$ for every $h'\in\, ]0, h]$ and every $n\in\Nbb$, the estimate 
	\eqref{D_t-bounds} continues to hold for any $h'\in \, ]0, h]$ with the same constant $c_{5}$. 
	From now on, w.l.o.g.\ we fix a constant $h \in \,]0,\min\{1,\mathbf{i}_{\mkern2mu\Tbb}\}[\,$ 
	for which Lemma~\ref{lem:hdef} holds. 
\end{remark}

We will now introduce some basic notions as in \cite[Sect.~3]{virag2000bnddegree}
to exploit the consequences of a positive anchored expansion constant.
But whereas Vir\'ag works with weighted volumes, ours refer to the cardinality of the sets in 
accordance with our previous definitions. 

\begin{definition}
	Let $q>0$ and let $\Tbf$ be a rooted tree for which each vertex has finite degree.
	\begin{enumerate}[\upshape(i)]
	\item
		The $q$-\emph{isolation} of a (possibly empty) finite vertex subset $S\subseteq\Tbf$ is given by
		\begin{equation}
			\Delta_q S \coloneqq \Delta_q^{\Tbf} S \coloneqq q\vv{S}-\ev{\partial S}.
		\end{equation}
		We will omit the superscript $\Tbf$ when there is no danger of confusion.
	\item
		We say that a finite vertex subset $S \subseteq \Tbf$ is $q$-\emph{isolated} whenever 
		\begin{equation}
			\Delta_q S >0.
		\end{equation}
	\item
		A (possibly empty) finite vertex subset set $S\subseteq \Tbf$ is called a $q$-\emph{isolated core} 
		of $\Tbf$ whenever 
		\begin{equation}
			\Delta_q S>\Delta_q A \quad \text{for every } A\subsetneq S.
		\end{equation}
	\item
		We write $A_{q} \coloneqq A_q(\Tbf)$ for the union of all $q$-isolated cores of $\Tbf$ and call 
		any connected component of $A_q$ a ($q$-)\emph{island}. The complement $\Tbf\setminus A_q$ is called 
		the ($q$-)\emph{oceans}.
	\end{enumerate}
\end{definition}

\begin{remark}
	\begin{enumerate}[(i)]
	\item
		\label{core-isolated}
		A non-empty $q$-isolated core is itself $q$-isolated because the subset $A$ in the definition can 
		be chosen as the empty set with $q$-isolation $\Delta_{q}\emptyset=0$.
	\item 
		\label{core-connected}
		Every connected component of a non-connected $q$-isolated core is itself a $q$-isolated core. 
		This follows from the additivity of the $q$-isolation w.r.t.\ connected components and by choosing 
		the subset $A$ to be the union of a proper subset of one connected component together with 
		all other connected components.
	\item
		It will turn out that the $q$-islands act as traps for the random walk and thus prevent us from obtaining 
		suitable heat-kernel bounds. Restricting the random walk to the $q$-oceans will allow us to 
		benefit from non-anchored, i.e.\ global, isoperimetric constants. 
	\end{enumerate}
	\label{core-rem}
\end{remark}

\noindent 
The definition \eqref{TTbounds} of the anchored expansion constant for Galton--Watson trees 
$\mathbf{i}_{\mkern2mu\Tbb}$ and Lemma~\ref{det-VolumeBound} directly imply

\begin{corollary}\label{VolumeBound}
	Let $q \in \,]0, \mathbf{i}_{\mkern2mu\Tbb}[\,$. Then, for $G$-almost every $\Tbf\in\Tbb$, every 
	$q$-island of $\Tbf$ has only finitely many vertices and thus is itself a $q$-isolated core of $\Tbf$. 
\end{corollary}

\noindent
In Vir\'ag's deterministic proof, 
\cite[Lemma 3.5]{virag2000bnddegree} is crucial. It is in this lemma where the dependence on the 
graph of the initial time $t(\Gbf)$ in the heat-kernel estimate \eqref{vir-est} originates from, 
thus impeding its use for an annealed bound. Moreover, this lemma is also one of the instances where 
boundedness of the vertex degrees plays a role. To avoid these shortcomings we pursue a probabilistic 
approach by excluding a third event of bad trees. 
This is the content of the next lemma whose formulation requires  to introduce several more notions. 

Let $\Tbf\in\Tbb$ and $q>0$ be fixed.  
A \emph{bridge structure} interconnecting a vertex set $S\subset\Tbf$ is a set of vertices 
$B\subset\Tbf$ such that $B\cup S$ is a connected set. 
A \emph{bridge} connecting two vertex sets $S_1,S_2 \subset\Tbf$ is a vertex set $B\subset\Tbf$ 
such that $B\cup S_1\cup S_2$ has a connected component intersecting both $S_1$ and $S_2$. 
We define the $q$-\emph{length} of a bridge $B \subset\Tbf$ by 
\begin{equation}
	q\text{-length}(B) \coloneqq |B\setminus A_q|,
\end{equation}
that is, the number of vertices of $B$ belonging to the $q$-oceans of $\Tbf$. 
Given a vertex set $S\subset\Tbf$ and a vertex $v\in\Tbf$, we define their \emph{$q$-distance} by
\begin{equation}
	\dist_q(v,S) \coloneqq \LEFTRIGHT\{.{
		\begin{array}{@{\,}lc}  0\,, & v \in S, \\[.3ex]
			 \displaystyle 1+ \; \min_{\smash{\substack{\\[3ex] \text{bridges } B\subset\Tbf \\[.1ex] 
			 	\clap{\scriptsize$\text{ connecting } \{v\} \text{ and } S$}}}}
				\;\big\{q\text{-length}(B)\big\}, &  v\not\in S.  \end{array} }\rule[-5.5ex]{0pt}{3ex}
\end{equation}
As noted before, $q$-islands pose a problem for obtaining heat-kernel bounds. Given $t\in\Nbb$, 
the event 
\begin{align}
	\label{defHt}
		H_{q,t}^{0} \coloneqq  \Bigl\{\Tbf  \in\Tbb:\; 
		& \exists \text{ a finite union of }q\text{-islands }U_{q,t}=U_{q,t}(\Tbf)\subseteq\Tbf\;  	
			\text{with} \notag \\
		& {2^{\frac56}t^\frac{1}{3}} \leq q \vv{U_{q,t}}<\infty \text{ and } \exists 
			\text{ a bridge structure }B_{q,t} \notag \\
		& \text{interconnecting } \{o\}\cup U_{q,t} \text{ with }  \max_{v\in B_{q,t}}\dist(o,v)
			\leq t\Bigr\}
\end{align}
describes trees where these islands are too dominant and situated too close to the root, that is, 
reachable for the random walk in $t$ steps. The next lemma allows to exclude the particularly bad situation, where these islands are too close together and too close to the root with respect to the $q$-length. However, such control is only possible with a restriction on the growth in the relevant part of the tree.

\begin{lemma}
	\label{badlargeislands}
	For $t\in\Nbb$ and $z_t \in \Nbb \setminus \{1\}$ let 
	\begin{equation}
		\label{Mt-def}
		M_{t,z_{t}}  \coloneqq  \big\{\Tbf\in\Tbb:\; Z_{\Tbf}(x)\leq z_t-1 \;\forall x\in\mathbf{T}_{ot}\big\}
	\end{equation}
	be the event of trees whose numbers of offspring are bounded by the same constant $z_t-1$ for every vertex 
	up to generation $t$. Furthermore, we set 
	\begin{equation}
		\label{qdef}
		q \coloneqq \frac23 h,
	\end{equation}
	where $h$ is given by Remark~\ref{rem:h-small}, and define the subset 
	\begin{equation}\label{territoryineq}
		H_{t} \coloneqq H_{q,t,z_{t}} \coloneqq  \Bigl\{\Tbf  \in H_{q,t}^{0}:\; z_t 
			\frac{\big|(B_{q,t}\cup\{o\})\setminus A_q\big|}{\vv{U_{q,t}}}\leq\frac{h}{3} \Bigr\}
	\end{equation} 
	of trees from $H_{q,t}^{0}$, for which there exists a small (w.r.t.\ the $q$-length) bridge 
	structure connecting the bad $q$-islands with the root and among each other. 
	Then, we have  
	\begin{equation} 
		\label{estHt}
		G[M_{t,z_{t}}\cap H_t]\leq \exp(-c_{5} t^\frac13),
	\end{equation} 
	where $c_{5}>0$ is the constant from Lemma~\ref{lem:hdef}.
\end{lemma}

\begin{proof}
	We fix $t\in\Nbb$ and $q \coloneqq \frac23 h < \frac23$. 
	Let $\Tbf\in M_{t,z_{t}}\cap H_t\subseteq M_{t,z_{t}}\cap H_{q,t}^{0}$. Let $A$ be the union of all $q$-islands of 
	$\Tbf$ intersecting $\{o\}\cup B_{q,t}\cup U_{q,t}$. Thus, $U_{q,t}\subseteq A\subseteq A_q$ 
	and $A$ is itself a $q$-isolated core so that  $\frac{\ev{\partial A}}{\vv{A}}<q$. 
	We define the part 
	\begin{equation}
		S \coloneqq  (\{o\}\cup B_{q,t})\setminus A_q
		= (\{o\}\cup B_{q,t}\cup U_{q,t}) \setminus A_q
		= (\{o\}\cup B_{q,t}\cup U_{q,t}) \setminus A		
	\end{equation}
	of the bridge structure and the root not belonging to any $q$-island. 

	Since we assume $\Tbf\in H_t$, the definition \eqref{territoryineq} implies that
	\begin{equation}
		\label{territoryineq1}
		z_t\frac{\vv{S}}{\vv{A}}\leq z_t\frac{\vv{S}}{\vv{U_{q,t}}}
		\leq\frac{h}{3}.
	\end{equation} 
	Furthermore, we conclude from $A\cap S=\emptyset$ that
	\begin{equation}
		\label{ineq2}
		\frac{\ev{\partial(A\cup S)}}{|A\cup S|}\leq\frac{\ev{\partial A} +\ev{\partial S}}{\vv{A} 
			+ \vv{S}}
		\leq  \frac{q+z_t\frac{\vv{S}}{\vv{A}}}{1+\frac{\vv{S}}{\vv{A}}}.
	\end{equation}
	For the last inequality we used $\ev{\partial A}< q\vv{A}$ and $\ev{\partial S}\leq z_t\vv{S}$, 
	which follows from $\Tbf\in M_{t,z_{t}}$ and that the bridge structure has maximal graph distance $t$ 
	to the root.

	For $0<q<1< z_{t}$, the elementary estimate 
	\begin{equation}
		\label{ineq3}
 		\frac{q+z_ta}{1+a}\leq\frac{q+z_tb}{1+b}   
	\end{equation}
	holds for any $0\leq a\leq b$. 
	The inequality \eqref{territoryineq1} allows to apply \eqref{ineq3} to \eqref{ineq2} 
	with $a=\frac{\vv{S}}{\vv{A}}$ and $b=\frac{h}{3 z_t}$, yielding
	\begin{equation}
		\label{est}
		\frac{\ev{\partial(A\cup S)}}{|A\cup S|}\leq\frac{q+\frac{h}{3}}{1+\frac{h}{3z_t}}<h
	\end{equation}
	by the definition \eqref{qdef} of $q$.
	To summarise we note that
	\begin{equation}
		K\coloneqq A\cup S = A \cup \{o\}\cup B_{q,t}\cup U_{q,t}
	\end{equation}
 	contains the root $o$, is connected, has finite 
	volume $|K| \ge |U_{q,t}| \ge 2^{\frac56}t^{\frac13}/q$ and satisfies \eqref{est}. Hence,  
	$\Tbf\in D_{\lfloor 2^{\frac56}t^{\frac13}/q \rfloor}$ and the claim follows from 
	\eqref{D_t-bounds} and $\lfloor 2^{\frac56}t^{\frac13}/q \rfloor \ge t^{\frac13}$ for $t\in\Nbb$.
\end{proof}


\section{Markov-kernel estimates for an effective random walk}

It is well known, see e.g.\ \cite{DiaconisStroock.1991.GeometricBounds, MR1743100, peres2016prob},  
that (non-anchored or global) iso\-perimetric-type inequalities on a graph imply 
bounds on the Markov and heat kernel for a random walk defined on this graph. 
And it is the $q$-oceans which give rise to such isoperimetric ratios that are bounded away from zero.
Therefore it will be the goal to derive the desired bound on the 
return probability from the properties of the random walk on the 
$q$-oceans. The latter are not connected in general, however. Therefore we will follow ideas of  
\cite[Sect.\ 3]{virag2000bnddegree} and construct an effective (a.k.a.\ induced)
Markov chain on the $q$-oceans which can jump over the $q$-islands. This Markov chain turns out to be 
a random walk on a connected weighted graph, which we introduce next.

We consider an infinite graph $\mathbf{G}$ together with a symmetric weight function
$w\colon \mathbf{G}\times \mathbf{G}\to [0,\infty[\,$, i.e., $w(x,y)=w(y,x)$ for all 
$x,y\in\mathbf{G}$. We assume finite vertex weights
\begin{equation}
	\label{defw}
	w(x) \coloneqq  \sum_{y\in\mathbf{G}}w(x,y)<\infty
\end{equation}
for every $x\in \mathbf{G}$ and that any two vertices $x,y\in\Gbf$, $x\neq y$, are connected by a path 
with strictly positive weights, that is, there exists $N\in\Nbb$ and $x_{n} \in \Gbf$, 
$n\in \{0,1, \ldots, N\}$, with $x_{0}= x$, $x_{N}=y$ and 
\begin{equation}
	\label{Gconn}
	\prod_{n=1}^{N} w(x_{n-1},x_{n}) >0.
\end{equation} 
The pair $(\mathbf{G},w)$ is called a \emph{weighted graph} and referred to as \emph{connected} 
due to \eqref{Gconn}.

For $\emptyset\neq S\subseteq \mathbf{G}$ finite, we define its $w$-weighted volume by
\begin{equation}
	\label{wvol}
	|S|_{w} \coloneqq   \sum_{x\in S}w(x)
\end{equation}
and its $w$-weighted size of the edge boundary by 
\begin{equation}
	\label{wbdy}
	|\partial S|_{w}  \coloneqq   \sum_{x\in S,\; y\in S^c}w(x,y).
\end{equation}
Following \cite[Sect.~6.1]{peres2016prob}, we introduce
\begin{definition}
	The \emph{edge-isoperimetric constant} of the weighted graph $(\mathbf{G},w)$ is given by
	\begin{equation}\label{qw}
		Q_w \coloneqq   \inf\Bigl\{ \frac{|\partial S|_{w}}{|S|_{w}}:\;\emptyset\neq S\subset \mathbf{G}
			\text{ finite} \Bigr\} \leq 1.
	\end{equation}
\end{definition}

\begin{remark}
	The edge-isoperimetric constant \eqref{qw} resembles the definition of the anchored expansion 
	constant \eqref{T-exp-const} but without the 
	anchor and for weighted graphs. However, in the absence of the anchor, the edge-isoperimetric 
	constant is typically zero for the realisations of a Galton--Watson tree when 
	equipped with the 
	canonical weights \eqref{can-weights} below.
\end{remark}

So far we have dealt with the simple random walk on unweighted (tree) graphs. Next, we generalise the notion to weighted graphs. The two notions will be related in the paragraph after Theorem~\ref{MarkovKernel}.
The \emph{standard} random walk on the weighted graph $(\mathbf{G},w)$ is defined by its 
transition probabilities $p(x,y): = \frac{w(x,y)}{w(x)}$ for moving from $x\in \mathbf{G}$ to 
$y\in \mathbf{G}$ within one time step. Furthermore, the 
corresponding symmetric Markov kernel $\mathbf{P}_{\mathbf{G},w}$ on the weighted real Hilbert space 
$\ell^2(\mathbf{G},w)$ is given by 
$(\mathbf{P}_{\mathbf{G},w}\psi)(x)  \coloneqq  \sum_{y\in\mathbf{G}}p(x,y)\psi(y)$ for every 
$\psi\in\ell^2(\mathbf{G},w)$ and every $x\in\mathbf{G}$. 
The weighted Hilbert space $\ell^2(\mathbf{G},w)$ is equipped with the $w$-weighted inner product 
$\langle\mkern2mu\pmb\cdot\mkern2mu , \mkern2mu\pmb\cdot\mkern2mu\rangle_{\mathbf{G},w}$, defined by 
$\langle\psi , \phi\rangle_{\mathbf{G},w} \coloneqq \sum_{x\in\mathbf{G}} w(x)\psi(x)\phi(x)$ for every 
$\psi,\phi\in \ell^2(\mathbf{G},w)$. 
Then we have $p(x,y) =\langle{1}_{\{x\}}, \mathbf{P}_{\mathbf{G},w} {1}_{\{y\}}\rangle_{\mathbf{G}}$
for every $x,y\in\Gbf$, where ${1}_{S}$ is the indicator 
function of a vertex subset $S\subseteq\mathbf{G}$ and 
$\langle\mkern2mu\pmb\cdot\mkern2mu , \mkern2mu\pmb\cdot\mkern2mu\rangle_{\mathbf{G}}$ 
is the unweighted $\ell^2$-inner product given by 
$\langle\psi , \phi\rangle_{\mathbf{G}}  \coloneqq \sum_{x\in\mathbf{G}}\psi(x)\phi(x)$ for every 
$\psi,\phi\in \ell^2(\mathbf{G})$. 

We recall that if the operator norm of the Markov kernel $\mathbf{P}_{\mathbf{G},w}$ is strictly
smaller than $1$, then this implies an upper bound on the heat kernel of the standard random walk 
which decays exponentially in time \cite[Prop.\ 6.6]{peres2016prob}. The next theorem, 
which we quote without proof, relates this criterion to a positive edge-isoperimetric constant.

\begin{theorem}[\protect{\cite[Thm.~6.7]{peres2016prob}}]
	\label{MarkovKernel}
	We consider the standard random walk on the connected infinite weighted graph $(\mathbf{G},w)$ 
	with edge-isoperimetric constant $Q_w$. 
	Then, its Markov kernel $\mathbf{P}_{\mathbf{G},w}$ fulfils
	\begin{equation}
		\label{markkbnd}
		\|\mathbf{P}_{\mathbf{G},w}\|_{\mathbf{G},w}\leq \sqrt{1-Q_w^2}\leq 1-\frac{Q_w^2}{2},
	\end{equation}
	where $\|\pmb\cdot\|_{\mathbf{G},w}$ denotes the operator norm on the Banach space of bounded 
	linear operators on $\ell^2(\mathbf{G},w)$.
\end{theorem}

In the rest of this section, we consider an infinite rooted tree $\Tbf$ 
for which every vertex $x\in\Tbf$ has a finite vertex 
degree $\deg(x) \coloneqq \deg_{\Tbf}(x) < \infty$. The simple random walk 
$\{X_t\}_{t\in\Nbb_0}$ on $\Tbf$ coincides with the standard random walk on the weighted 
graph $(\Tbf, w_{\mathrm{srw}})$ with canonical edge weights equal to 
\begin{equation}
	\label{can-weights}
	w_{\mathrm{srw}}(x,y) \coloneqq \LEFTRIGHT\{.{\begin{array}{@{\mkern3mu}ll} 1, & \text{if $\{x,y\}$ is an edge of }
			\Tbf, \\ 0, & \text{otherwise}.  \end{array}}
\end{equation}
Thus, the corresponding vertex weight is given by $w_{\mathrm{srw}}(x) \coloneqq \deg(x)$ for every $x\in\Tbf$ 
according to \eqref{defw}. 
We write $\mathbf{P}\coloneqq \mathbf{P}_{\Tbf, w_{\mathrm{srw}}}$ for the associated Markov operator on $\ell^{2}(\Tbf)$, 
which is symmetric on $\ell^{2}(\Tbf, w_{\mathrm{srw}})$.

In Lemma~\ref{wic} we will establish that the $q$-oceans 
\begin{equation}
	\Tbf_q \coloneqq \Tbf\setminus A_{q}
\end{equation}
exhibit a non-anchored isoperimetric inequality.
Therefore we would like to apply 
Theorem~\ref{MarkovKernel} to the standard random walk on the weighted 
graph $(\Tbf_{q}, w_{\mathrm{srw}})$. But the theorem requires a connected weighted graph to 
yield a non-trivial result. Therefore we will construct a weight function $w_{q}$ which ensures 
that $(\Tbf_q,w_{q})$ is connected and 
such that the standard random walk $\{W_{t}\}_{t\in\Nbb_{0}}$
on $(\Tbf_q,w_{q})$ behaves like the simple random walk $\{X_{t}\}_{t\in\Nbb_{0}}$ on $\Tbf$ 
if the latter is only observed on the $q$-oceans $\Tbf_{q}$. 
This standard random walk $\{W_{t}\}_{t\in\Nbb_{0}}$
on $(\Tbf_q,w_{q})$ is often referred to as the induced Markov chain of $\{X_t\}_{t\in\Nbb_0}$ on 
$\Tbf_q$ and is specified in

\begin{definition}
	\label{def:Tq}
	Let $q >0$.
	\begin{enumerate}[\upshape(i)]
	\item
		We write 
		\begin{equation}
			\tau_{S} \coloneqq \inf\big\{t \in\Nbb: X_{t}\in S\big\}  \in \Nbb \cup \{\infty\}
		\end{equation}
		for the \emph{first hitting time} after zero of a vertex subset
		$S \subseteq\Tbf$ by the simple random walk $\{X_t\}_{t\in\Nbb_0}$ on $\Tbf$. We also introduce the 
		abbreviation $\tau_{\mathrm{oc}}\coloneqq \tau_{\Tbf_{q}}$ for the first hitting time of the 
		$q$-oceans. 
	\item
		The edge weights of the weighted graph $(\Tbf_q,w_{q})$
		are given by
		\begin{equation}
			\label{wq}
			w_q(x,y) \coloneqq  w_{\mathrm{srw}}(x)P^{\Tbf}_x[X_{\tau_{\mathrm{oc}}}=y]
		\end{equation}
		for all vertices $x,y\in\Tbf_q$. Accordingly, $P^{\Tbf}_x[X_{\tau_{\mathrm{oc}}}=y]$ is the 
		probability that the simple random walk on $\Tbf$ ends up at $y$ 
		when it first touches the oceans $\Tbf_q$ again after having left its starting point $x$.
	\item
		We write $P^{\Tbf_{q}}_{x}$ for the probability measure of the standard random walk 
		$\{W_{t}\}_{t\in\Nbb_{0}}$ 
		on $(\Tbf_q,w_{q})$, which starts at $x\in\Tbf_{q}$. 
		We use the symbol $\mathbf{P}_{q} \coloneqq \mathbf{P}_{\Tbf_{q}, w_{q}}$ for the associated Markov operator 
		on $\ell^{2}(\Tbf_{q})$, which is symmetric on $\ell^{2}(\Tbf_{q},w_{q})$.
	\end{enumerate}
\end{definition}

\noindent
The following properties hold.

\begin{remark}
\begin{enumerate}[(i)]
\item
	For every $x,y\in \Tbf_q$ we have 
	\begin{equation}
		\label{wqsym}
		w_{q}(x,y) = w_{q}(y,x), 
	\end{equation}
	that is, $w_{q}$ is symmetric and, 
	hence, it is indeed an edge-weight function. This follows from time reversibility of the simple 
	random walk on $\Tbf$: Consider a path 
	$X_{0}=x, X_{1}=x_{1}, \ldots, X_{n}=x_{n}, X_{n+1}=y$ along the edges in $\Tbf$, where $n\in\Nbb_{0}$ and $x_{j}\in A_{q}$ for
	$j=1,\ldots,n$, which contributes to the probability 
	$P^{\Tbf}_x[X_{\tau_{\mathrm{oc}}}=y]$. This path has probability 
	$\frac{1}{w_{\mathrm{srw}}(x)} \prod_{j=1}^{n} \frac{1}{w_{\mathrm{srw}}(x_{j})}$. 
	It corresponds uniquely to a time-reversed path 
	$X_{0}=y, X_{1}=x_{n}, \ldots, X_{n}=x_{1}, X_{n+1}=x$ contributing to
	$P^{\Tbf}_y[X_{\tau_{\mathrm{oc}}}=x]$ with probability 
	$\frac{1}{w_{\mathrm{srw}}(y)} \prod_{j=1}^{n} \frac{1}{w_{\mathrm{srw}}(x_{n-j+1})}$. 
	The same holds vice versa and proves \eqref{wqsym}.
\item 
	\label{ocean-weight}
	For every $x,y\in \Tbf_q$ we have 
	\begin{equation}
		\label{wqedge}
		w_q(x,y) \ge w_{\mathrm{srw}}(x,y),
	\end{equation}
	where strict inequality can only occur if both $x$ and $y$ belong to the outer vertex boundary 
	$\partial_{\mathrm{out}}C \coloneqq \{\wtilde{x} \in \Tbf: \dist(\wtilde{x}, C)=1\}$ of 
	the same $q$-island $C\in A_{q}$. 
	Indeed, if there is no edge in $\Tbf$ between $x$ and $y$, then $w_{\mathrm{srw}}(x,y)=0$,
	and the inequality 
	is trivial. If there is an edge in $\Tbf$ between $x$ and $y$ then there exists a one-step path from $x$ to $y$ 
	with $\tau_{\mathrm{oc}}=1$ and $w_{\mathrm{srw}}(x)P^{\Tbf}_x[X_{1}=y] =1$. 
	If both $x,y \in \partial_{\mathrm{out}}C$
	there exists a path from $x$ to $y$ lying entirely in the $q$-island $C$ except for the two 
	endpoints $x$ and $y$. In this case, 
	$P^{\Tbf}_x[X_{\tau_{\mathrm{oc}}}=y \text{ and } \tau_{\mathrm{oc}} >1] >0$ gives rise to a 
	positive contribution to $w_{q}(x,y)$ whereas $w_{\mathrm{srw}}(x,y)=0$ because $\Tbf$ is a tree. 
\item
	Let $C\in A_{q}$ be a $q$-island and assume that $x\in\partial_{\mathrm{out}}C$. Then
	\begin{equation}
		w_{q}(x,x) >0.
	\end{equation}
	If $x\in \Tbf_{q}$ is not adjacent in $\Tbf$ to any $q$-island, then $w_{q}(x,x) =0$, 
	as follows from \eqref{wq}.
	In other words, the standard random walk on the weighted graph has a non-zero probability to stay at a vertex $x$ instead of jumping to another vertex $y\neq x$, but this may happen only at vertices $x$ that are adjacent  in $\Tbf$ to some $q$-island.
\item 	
	For every $x\in\Tbf_q$ we have  $\sum_{y\in\Tbf_q} P_x[X_{\tau_{\mathrm{oc}}}=y]=1$ and thus
	\begin{equation}
		\label{wqvertex}
		w_q(x) \coloneqq \sum_{y\in\Tbf_q} w_q(x,y) = w_{\mathrm{srw}}(x).
	\end{equation}
\item 
	We claim that the probability for an arbitrary vertex $y \in\Tbf_{q}$ to be reachable  
	by the simple random walk $\{X_{t}\}_{t_{\in\Nbb_{0}}}$ on $\Tbf$, which starts at $x\in\Tbf_{q}$,
	is the same as for the 
	standard random walk $\{W_{t}\}_{t_{\in\Nbb_{0}}}$ on $(\Tbf_{q},w_{q})$, that is,
	\begin{equation}
		\label{srw=q}
		P^{\Tbf}_{x}[\exists\, t\in\Nbb: X_{t} =y] = P^{\Tbf_{q}}_{x}[\exists\, t\in\Nbb: W_{t} =y].
	\end{equation}
	Before we prove \eqref{srw=q} we note that it immediately implies
	\begin{equation}
		\label{srw=q0}
		P^{\Tbf}_{x}[\exists\, t\in\Nbb_{0}: X_{t} =y] 
		= P^{\Tbf_{q}}_{x}[\exists\, t\in\Nbb_{0}: W_{t} =y].
	\end{equation}
	To prove \eqref{srw=q} let $\{\sigma_{n}\}_{n\in\Nbb_{0}}$ be 
	the strictly increasing sequence of stopping times which are uniquely defined by $\sigma_{0}\coloneqq0$, 
	$\sigma_{n} \coloneqq \sigma_{n-1}+1$ if both $X_{\sigma_{n-1}}, X_{\sigma_{n-1}+1} \in \Tbf_{q}$ 
	for $n\in\Nbb$
	and the property that $X_{t} \in A_{q}$ if and only if	$\sigma_{n-1} < t < \sigma_{n}$ 
	for some $n\in\Nbb$. 
	We infer that 
	\begin{align}
		\label{srw=q-left}		
		P_{x}^{\Tbf}\big[ \exists\, t\in\Nbb &: X_{t}=y\big] \notag \\
		& = P_{x}^{\Tbf}\big[ \exists\, t\in\Nbb: X_{\sigma_{t}}=y\big] \notag \\
		& = \sum_{t\in\Nbb} P_{x}^{\Tbf}\Big[ X_{\sigma_{t}}=y \text{ ~and~ } X_{\sigma_{n}} \in 
					\Tbf_{q}\setminus \{y\} \; \forall  n=1,\ldots, t-1 \Big]  \notag \\
		& = \sum_{t\in\Nbb} \; \sum_{\substack{y_{1},\ldots,y_{t-1} \\ \in \Tbf_{q}\setminus \{y\}}}
					\;\prod_{n=1}^{t} P_{y_{n-1}}^{\Tbf} \big[ X_{\sigma_{1}} = y_{n} \big],  
	\end{align}	
	where $y_{0}\coloneqq x$, $y_{t}\coloneqq y$ and we used the strong Markov property for the last equality. 
	Now $\sigma_{1}=\tau_{\mathrm{oc}}$ and the probability in the 
	last line of \eqref{srw=q-left} is equal to 
	\begin{equation}
		\label{srw=q-1-step}
		P_{y_{n-1}}^{\Tbf}[X_{\tau_{\mathrm{oc}}} = y_{n}] = \frac{w_{q}(y_{n-1},y_{n})}{w_{q}(y_{n-1})}
		= P_{y_{n-1}}^{\Tbf_{q}} [W_{1} = y_{n}],
	\end{equation}
	where the first equality follows from \eqref{wq} and \eqref{wqvertex}, and the second equality from the definition
	of the standard random walk. Inserting \eqref{srw=q-1-step} 
	into \eqref{srw=q-left} and using the Markov property for $\{W_{t}\}_{t\in\Nbb_{0}}$, we infer 
	\begin{equation}
		P_{x}^{\Tbf}\big[ \exists\, t\in\Nbb : X_{t}=y\big] 
		 = \sum_{t\in\Nbb} P_{x}^{\Tbf_{q}}\Big[ W_{t}=y \text{ ~and~ } W_{n} \in 
					\Tbf_{q}\setminus \{y\} \; \forall  n=1,\ldots, t-1 \Big]
	\end{equation}
	so that \eqref{srw=q} follows. 
\end{enumerate}
\end{remark}

The next lemma establishes that when viewing the $q$-oceans as the weighted graph $(\Tbf_q,w_q)$, they 
satisfy a non-anchored isoperimetric inequality.

\begin{lemma}
	\label{wic}
	Let $q \in \,]0,2]$ and let $\Tbf$ be a rooted tree with $\deg(x) < \infty$ for every $x\in\Tbf$. 
	Then, the weighted graph $(\Tbf_q,w_q)$ has edge-isoperimetric constant 
	\begin{equation}\label{qwq}
		Q_{w_q}\geq \frac{q}{q+2}. 
	\end{equation}
\end{lemma}

\begin{proof} 
	Let $\emptyset\neq S \subseteq \Tbf_q$ be a finite vertex subset. First we will reduce weighted 
	(edge) volumes to unweighted ones. To this end, let $E_{S}$ be the edge set of the subgraph in $\Tbf$ induced by $S$, i.e.\ the set of edges in $\Tbf$ which connect only vertices within $S$. The relation  \eqref{wqvertex} implies the first equality in  
	\begin{equation}
		|S|_{w_{q}} 
		= |S|_{w_{\mathrm{srw}}} 
		= \sum_{x\in S} \deg_{T}(x) = 2 |E_{S}| + |\partial^{\Tbf}S|
		= 2(|S| -1) + |\partial^{\Tbf}S| ,
	\end{equation}
	and the last equality holds because $\Tbf$ is a tree.
%
%
The inequality 
	\eqref{wqedge} implies that $|\partial^{\Tbf_{q}}S|_{w_{q}} 
		\ge |\partial^{\Tbf_{q}}S|_{w_{\mathrm{srw}}} = |\partial^{\Tbf_{q}}S|$. 
	Now, let $C$ be a (possibly empty) $q$-isolated core containing all vertices in $A_q$ which are adjacent 
	to $S$.	When $A_q$ is removed from $\Tbf$, the volume $|\partial^{\Tbf}S|$
	of the edge boundary of $S$ in $\Tbf$ decreases by the 
	number of edges connecting $S$ with $C$. Hence, we obtain 
	$|\partial^{\Tbf_{q}}S| = |\partial^{\Tbf} S| - |\partial^{\Tbf} S\cap\partial^{\Tbf} C|$. 
	Altogether, we arrive at the estimate
	\begin{align}
		\label{iso-bound}
		\frac{q}{2} \, |S|_{w_q} - \frac{2+q}{2}\, |\partial^{\Tbf_q}S|_{w_q} 
		&\le \Delta^{\Tbf}_{q}S + \frac{2+q}{2}\, |\partial^{\Tbf} S\cap\partial^{\Tbf} C| \notag\\
		&= \Delta_{q}^{\Tbf}(S\cup C) - \Delta^{\Tbf}_{q}C 
			- \frac{2-q}{2} \,|\partial^{\Tbf} S\cap\partial^{\Tbf} C| \notag \\
		&\le \Delta_{q}(S\cup C) - \Delta_{q}C ,		
	\end{align}
	where the equality results from an application of \eqref{DeltaDisjSplit} and the last inequality 
	from $q \le 2$. For notational simplicity we dropped the superscript $\Tbf$ from the $q$-isolations in the last line of \eqref{iso-bound}.

 Since $S\subseteq \Tbf\setminus A_{q}$, the vertex 
	subset $S\cup C$ cannot be a $q$-isolated core in $\Tbf$. By definition, there must exist a (possibly empty) 
	vertex subset $B\subsetneq S \cup C$ with 
	\begin{equation}\label{iso1}
		\Delta_q(S\cup C)\leq \Delta_q B.
	\end{equation}
	W.l.o.g.\ we choose this vertex subset $B$ to be minimal in the sense that no proper subset of $B$ 
	has the property \eqref{iso1}. In other words, for every $\wtilde{B} \subsetneq B$, we must have 
	$\Delta_q(S\cup C) > \Delta_{q}\wtilde{B}$. Together with \eqref{iso1}, this means that $B$ is a 
	$q$-isolated core, whence $B \subseteq C$. Applying Lemma~\ref{DeltaBound} with $A=B$ and the 
	$q$-isolated core $S=C$, yields $\Delta_{q}(B) \le \Delta_{q}(B\cup C) = \Delta_{q}(C)$. Combining this 
	inequality with \eqref{iso1}, yields
	\begin{equation}\label{iso2}
		\Delta_q(S\cup C)\leq \Delta_q C.
	\end{equation}
	Now, \eqref{iso2} and \eqref{iso-bound} imply
	\begin{equation}
		\frac{q}{2} \, |S|_{w_q} - \frac{2+q}{2}\, |\partial^{\Tbf_q}S|_{w_q} \le 0,
	\end{equation}
	and the claim follows.
\end{proof}

Switching between the weighted trees $(\Tbf,w_{\mathrm{srw}})$ 
and $(\Tbf_q,w_q)$ will not only be done with the help of \eqref{srw=q} but also on the level of the 
corresponding weighted Hilbert spaces. 

\begin{definition}
	Let $q>0$ and let $\Tbf$ be an infinite rooted tree with $\deg(x) < \infty$ for every $x\in\Tbf$.
	We introduce the restriction map
	\begin{equation}
		\label{rho}
		\rho_{\Tbf} \colon
		\begin{array}{rcl} \ell^2(\Tbf, w_{\mathrm{srw}}) & \to & \ell^2(\Tbf_q,w_q) \\[1ex]
			(\psi_x)_{x\in \Tbf} & \mapsto & (\psi_{x})_{x\in \Tbf_q}
		\end{array}
	\end{equation}
	and its adjoint, the embedding 
	\begin{equation}
		\label{rhoa}
		\rho^*_{\Tbf}\colon \begin{array}{rcl} \ell^2(\Tbf_q, w_q) & \to 
			& \ell^2(\Tbf, w_{\mathrm{srw}}) \\[1ex]
			(\varphi_x)_{x\in \Tbf_q} & \mapsto & (\wtilde{\varphi}_x)_{x\in \Tbf}
		\end{array}, \quad \text{where} \quad 
		\wtilde{\varphi}_{x}  \coloneqq  \LEFTRIGHT\{.{ \begin{array}{@{}cl} \varphi_{x}, & \text{if } 
		x\in \Tbf_q, \\[.5ex] 0, & \text{if } x\in  A_q. \end{array}}
	\end{equation}
	We drop the index $\Tbf$ in our notation for both maps, if the underlying tree is clear. 
	Both $\rho$ and $\rho^{*}$ have operator norm $1$ due to \eqref{wqvertex}.
\end{definition}

The next lemma estimates the probability for the random walk to enter a bad geometric region 
consisting of several $q$-islands. Since we are on a tree we are able to obtain an estimate which scales 
with the square root of the number of involved $q$-islands. Without the tree property, one would get a
scaling with the square root of the volume of the edge boundaries of the involved $q$-islands as in 
\cite{virag2000bnddegree}. 
The improved scaling for trees will be crucial when applying the lemma in the next section. 

\begin{lemma}\label{TimeToEscape}
	Let $q \in\,]0,1[\,$ and let $\Tbf$ be an infinite rooted tree with infinite $q$-oceans
	$|\Tbf_{q}| = \infty$. 
	Furthermore, we assume the existence of $z\in\Nbb\setminus\{1\}$ 
	such that the growth condition $\deg(x')\leq z$ holds for every 
	$x'\in \mathbf{T}_{q}$.
	Let $J\in\Nbb$ and let $\mathcal C\coloneqq\bigcup_{j=1}^J C_j\subseteq \Tbf$ be the union of 
	arbitrarily chosen 
	$q$-islands $C_j\in A_q$,	$j\in\{1,...,J\}$. We also fix a vertex $x\in \Tbf$ with 
	$\dist_q(x,\mathcal C)\geq n$ for some $n\in\Nbb$. 
	Then, 
	\begin{equation}
		P_x[\tau_{\mathcal C} < \infty] \leq \frac{18}{q^{2}} \, \Bigl(1-\frac{q^2}{9}\Bigr)^{\frac{n}{2} 
			- 1} \, (zJ)^{\frac12}.
	\end{equation}
\end{lemma}

\begin{proof}
	To begin with we will argue that we may conduct the proof 
	assuming w.l.o.g.\ that $x \in\Tbf_{q}$. Indeed, since $\dist_q(x,\mathcal C)\geq n$,
	we have $x\notin \mathcal C$. So suppose that $x\in A_q\setminus \mathcal C$. 
	Then there exists a $q$-island $C'\subseteq A_q\setminus \mathcal C$ such that $x\in C'$ 
	and we must have 
	$\dist_{q}(x,\mathcal C) \ge \max\{n,2\}$. In order to 
	reach $\mathcal C$, the simple random walk $(X_t)_{t\in\Nbb_0}$ on $\Tbf$ has to hit the outer 
	vertex boundary $\partial_{\mathrm{out}}C'$ before hitting $\mathcal C$. 
	Therefore the strong Markov property of $(X_t)_{t\in\Nbb_0}$ at the hitting time of 
	$\partial_{\mathrm{out}}C'$ implies
	\begin{equation}\label{xtov}
		P_x[\tau_{\mathcal C} < \infty]  
		= E_x\Bigl[P_{X_{\tau_{\partial_{\mathrm{out}}C'}}}[\tau_{\mathcal C} < \infty]\Bigr]
		\leq \sup_{y\in\partial_{\mathrm{out}} C'}P_y[\tau_{\mathcal C} < \infty], 
	\end{equation}	
	where $E_{x}\coloneqq \int\d P_{x}$ is the probabilistic expectation associated with $P_{x}$.
	Because of \eqref{xtov}, $\partial_{\mathrm{out}}C' \subseteq\Tbf_{q}$ and 
	$\dist_q(y,\mathcal C) \ge \dist_q(x,\mathcal C) -1$ for all $y \in \partial_{\mathrm{out}} C'$, 
	which holds due to $C'\subseteq A_q$, it is sufficient to consider $x\in\Tbf_{q}$ 
	with $\dist_{q}(x,\mathcal C) \ge \max\{n-1,1\}$ in the rest of this proof.  

	So, let us fix $x\in\Tbf_{q}$ with $\dist_{q}(x,\mathcal C) \ge \max\{n-1,1\}$. 
	Since $\Tbf$ is a tree and $\mathcal C$ consists of $J$ connected components there exists a subset 
	$V \subseteq \partial_{\mathrm{out}}\mathcal C \subseteq \Tbf_{q}$ of the outer vertex 
	boundary of $\mathcal C$ with $|V| \le J$
	and such that the 
 	simple random walk $(X_t)_{t\in\Nbb_0}$ has to pass a vertex from $V$ in the last step on his 
	way from $x$ before hitting $\mathcal C$ for the first time. 
	Thus, 
	we infer that
	\begin{equation}
		\label{first-tau-bound}
		P_x[\tau_{\mathcal C} < \infty] 
		\le \sum_{y \in V} P_{x} \big[ \exists\, t\in\Nbb_{0}: \; X_t =y \big] .
	\end{equation}
	Applying \eqref{srw=q0} and the union bound to the probability on the right-hand side of 
	\eqref{first-tau-bound}, rewriting it in terms of the Markov operator $\mathbf{P}_{q}$ 
	and then switching first from the unweighted Hilbert space 
	$\ell^{2}(\Tbf_{q})$ to the weighted Hilbert space $\ell^{2}(\Tbf_{q}, w_{q})$ and finally to 
	$\ell^{2}(\Tbf, w_{\mathrm{srw}})$ with the embedding $\rho^{*}$ and using \eqref{wqvertex}, 
	we obtain
	\begin{equation}
		\label{srw-q-change-done}
		P_{x} \big[ \exists\, t\in\Nbb_{0}: \; X_t =y \big]
		\le \frac{1}{w_{\mathrm{srw}}(x)} \, \sum_{t\in\Nbb_{0}} \langle 1_{\{x\}}, \rho^{*} 
			\mathbf{P}_{q}^{\mkern2mu t} \rho	1_{\{y\}} \rangle_{\Tbf, w_{\mathrm{srw}}}.
	\end{equation}
	We deduce from \eqref{first-tau-bound} and \eqref{srw-q-change-done}
	that 
	\begin{equation}
		\label{second-tau-bound}
		P_x[\tau_{\mathcal C} < \infty] 
		\le \frac{1}{w_{\mathrm{srw}}(x)} \, \sum_{t\in\Nbb_{0}}
				\langle 1_{\{x\}}, \rho^{*} \mathbf{P}_{q}^{\mkern2mu t} \rho	1_{V} 
				\rangle_{\Tbf, w_{\mathrm{srw}}}.
	\end{equation}
	Since $\dist_q(x,\mathcal C)\geq \max\{n-1,1\}$, the random walk needs at least 
	$\nu\coloneqq\max\{n-2,0\}$ 
	steps on the infinite connected weighted graph $\Tbf_q$ to reach 
	$V \subseteq \partial_{\mathrm{out}}\mathcal C$ from $x$ and every term in the $t$-series in 
	\eqref{second-tau-bound} with $t<\nu$ vanishes. 
	We note that 
	$\sum_{t=\nu}^{\infty} \mathbf{P}_{q}^{\mkern2mu t}= \mathbf{P}_{q}^{\mkern2mu \nu} \mathbf{K}_q$, 
	where the Green kernel $\mathbf{K}_q \coloneqq  \sum_{t\in\Nbb_{0}}\mathbf{P}_q^{\mkern2mu t}$ exists in 
	operator norm in the space of bounded operators on $\ell^{2}(\Tbf_{q},w_{q})$ and satisfies the norm 
	estimate
	\begin{equation}\label{Ks}
		\|\mathbf{K}_q\|_{\Tbf_{q},w_q} 
		\leq \frac{1}{1-\|\mathbf{P}_q\|_{\Tbf_{q},w_q}} 
		\leq \frac{18}{q^2}
	\end{equation}
	because $\|\mathbf{P}_q\|_{\Tbf_{q},w_q} \le (1- q^{2}/9)^{1/2} \le 1- q^{2}/18$
	due to Theorem~\ref{MarkovKernel} and $Q_{w_{q}} \ge q/3$, which follows from Lemma~\ref{wic} and $q <1$. Accordingly, the $t$-series in 
	\eqref{second-tau-bound} can be written as 
	\begin{align}
		\label{Hilbertnorm1}
		\langle 1_{\{x\}}, \rho^*\mathbf{P}_q^{\mkern2mu \nu}	\mathbf{K}_q\rho 1_{V} 
			\rangle_{\Tbf,w_{\mathrm{srw}}}
		&\le \| 1_{\{x\}}\|_{\Tbf,w_{\mathrm{srw}}} \|\mathbf{P}_q\|_{\Tbf_{q},w_q}^\nu
				\|\mathbf{K}_q\|_{\Tbf_{q},w_q} \| 1_{V} 
				\|_{\Tbf,w_{\mathrm{srw}}} \notag \\
		& \le w_{\mathrm{srw}}(x)^{\frac{1}{2}} \Bigl(1-\frac{q^2}{9}\Bigr)^{\frac{\nu}{2}}\; 
				\frac{18}{q^2}	\, J^{\frac12} z^{\frac12}, 
	\end{align}
	where the first inequality relies on the Cauchy--Schwarz inequality and the fact 
	that the operator norms of $\rho$ and $\rho^{*}$ equal $1$.
	Now, the lemma follows from \eqref{second-tau-bound} and \eqref{Hilbertnorm1}.
\end{proof}


\section{Proof of Theorems~\ref{Decay} and \ref{thm}}
\label{sec:proofs}

The next theorem is our main technical result. It will allow us to prove Theorem~\ref{Decay} 
and Theorem~\ref{thm}.

\begin{theorem}\label{exactdecay} 
	Let $(z_t)_{t\in\Nbb}\subseteq \Nbb\setminus\{1,2\}$ be a sequence of 
	constants with $z_t=\;\scriptstyle\mathcal{O}$$(t^{\frac12})$ as $t\to\infty$. 
	Then, there exists an initial time $t_{0}\in\Nbb$ such that 
	\begin{equation}
		\label{edest}
		P^{\Tbf}_{o}[X_{t}=o] \leq \exp\bigg[- \frac{h^2}{144}\,\Big(\frac{t}{z_{t}^{2}}\Big)^{\frac13}\bigg] 
	\end{equation}
	for every $t \ge t_{0}$ and every $\Tbf\in (M_{t,z_{t}}\cap H_t^c) \setminus\mathcal{N}$. 
	Here, $h>0$ is given 
	by Remark~\ref{rem:h-small}, 
	the events $M_{t,z_{t}}$ and $H_{t}$ are introduced in 
	Lemma~\ref{badlargeislands} and the $G$-null set $\mathcal{N}$ is defined to be the union of 
	the $G$-null set 
	where $|\Tbf| = \infty$ fails with the $G$-null set where Corollary~\ref{VolumeBound} fails.
	We note that the initial time $t_{0}\in\mathbb{N}$ depends only on the given sequence 
	$(z_t)_{t\in\Nbb}$ and on $h$.
\end{theorem}

\noindent
Before we can prove Theorem~\ref{exactdecay} with the help of the results from the previous section, 
we have to to deal with the possibly unbounded number of offspring in the oceans of the tree 
$\Tbf \in M_{t,z_{t}}$ beyond the height $t$. 

\begin{definition}
	\label{def:tprime}
	Let $q \in \,]0,1[\,$, $t\in\Nbb$, $z_{t} \in\Nbb\setminus\{1\}$ and consider a tree 
	$\Tbf \in M_{t,z_{t}}$. 
	\begin{enumerate}[\upshape(i)]
	\item
		We construct recursively, starting from the root, an associated \emph{$q$-regularised tree}
		$\Tbf^{\mkern1mu q}$---not to be confused with $\Tbf_{q}$ from Definition~\ref{def:Tq}---by 
		\begin{equation}
			\label{T'offspring}
			Z_{\Tbf^{\mkern1mu q}}(x) \coloneqq \LEFTRIGHT\{.{\begin{array}{@{\,}ll@{}} Z_{\Tbf}(x), & \text{if } x\in \Tbf_{ot} 
				\text{ or} \\
				& \text{if } x\in C \text{ for some $q$-island } C \subseteq A_{q}(\Tbf) \text{ with } \dist(o, C) \le t,\\[.5ex]
				z_{t} -1, & \text{otherwise.}\end{array}
				}
		\end{equation}
		This means that $\Tbf_{ot} = \Tbf^{\mkern1mu q}_{ot}$ and that the regularised tree 
		$\Tbf^{\mkern1mu q}$ is homogenous from height $t+1$ onwards except at the vertices 
		of those $q$-islands of $\Tbf$ which have non-trivial intersection with $\Tbf_{ot}$ and extend also
		beyond the height $t$.	
	\item
		We write $\{X^{(q)}_{t}\}_{t\in\Nbb_{0}}$ for the simple random walk on $\Tbf^{q}$ and 
		$\tau^{(q)}_{S}$ for the first hitting time after zero of the vertex subset $S\subseteq \Tbf^{q}$ 
		by $\{X^{(q)}_{t}\}_{t\in\Nbb_{0}}$.
		The \emph{regularised weighted graph} $(\Tbf^{\mkern1mu q}_{q},w^{(q)}_{q})$, where 
		$\Tbf^{\mkern1mu q}_{q} \coloneqq \Tbf^{\mkern1mu q}\setminus A_{q}
			(\Tbf^{\mkern1mu q})$ are the $q$-oceans of $\Tbf^{q}$, 
		 is given as in Definition~\ref{def:Tq} 
		but with every reference to $\Tbf$ there replaced by $\Tbf^{\mkern1mu q}$, that is,
		\begin{equation}
			w^{(q)}_{q}(x,y)  \coloneqq  w^{(q)}_{\mathrm{srw}}(x)
			P^{\Tbf^{\mkern1mu q}}_x[X^{(q)}_{\tau^{(q)}_{\Tbf^{\mkern1mu q}_{q}}}=y]
		\end{equation}
		for every $x,y \in \Tbf^{\mkern1mu q}_{q}$ with $w^{(q)}_{\mathrm{srw}}(x) \coloneqq \deg_{\Tbf^{\mkern1mu q}}(x)$ being 
		the vertex degree of $x$ in $\Tbf^{\mkern1mu q}$. The standard random walk on 
		$(\Tbf^{\mkern1mu q}_{q},w^{(q)}_{q})$ will be denoted by $\{W^{(q)}_{t}\}_{t\in\Nbb_{0}}$.
	\end{enumerate}
\end{definition}

\begin{lemma}
	\label{prime}
	Let $0 <q' \le q < \min\{\mathbf{i}_{\mkern2mu\Tbb},1\}$, let $z_{t} \in\Nbb\setminus\{1,2\}$ 
	and consider an infinite tree 
	$\Tbf \in M_{t,z_{t}} \setminus\mathcal{N}$, with $\mathcal{N}$ being the null set from Theorem~\ref{exactdecay}. Then
	\begin{enumerate}[\upshape(i)]
	\item
		\label{prime-bdd}
		$Z_{\Tbf^{\mkern1mu q}}(x) \le z_{t}-1$ for all $x\in \Tbf^{\mkern1mu q}_{q}$.
	\item
		\label{prime-islands}
		The representation
		\begin{equation}
			\label{eq:prime-islands}
			\displaystyle	A_{q'}(\Tbf^{\mkern1mu q}) = \bigcup_{\substack{q'\text{-islands } 
			C\subseteq A_{q'}(\Tbf) : \;	\dist(o,C) \le t}} 
		\end{equation}
		holds, and we have $|\Tbf^{\mkern1mu q} \setminus A_{q'}(\Tbf^{\mkern1mu q})| =\infty$.
	\item
		\label{prime-prob}
		$w^{(q)}_{q}(x,y) = w_{q}(x,y)$ for every $x,y \in \Tbf_{ot} \setminus
		A_{q}(\Tbf) = \Tbf^{\mkern1mu q}_{ot} \setminus A_{q}(\Tbf^{\mkern1mu q})$.
	\end{enumerate}
\end{lemma}

\begin{proof}
	Part~\eqref{prime-bdd} holds by construction of $\Tbf^{\mkern1mu q}$ and because $\Tbf\in M_{t,z_{t}}$.

	As to Part~\eqref{prime-islands} we define 
	\begin{equation}
		\wtilde{A} \coloneqq \bigcup_{\substack{q'\text{-islands } 
			C\subseteq A_{q'}(\Tbf) : \;	\dist(o,C) \le t}} C
	\end{equation}
	and show two inclusions.\\
	``$\wtilde{A} \subseteq A_{q'}(\Tbf^{\mkern1mu q})$'' \quad
	The vertices of $\Tbf_{ot} \cup \wtilde{A}$ belong also to $\Tbf^{\mkern1mu q}$ and have the same degree in 
	$\Tbf^{\mkern1mu q}$ as in $\Tbf$. Therefore and since $\dist(o,C) \le t$ for every $q'$-island 
	$C \subseteq \wtilde{A}$, we infer that 
	\begin{equation}
		\label{q-isol-equal}
		\Delta_{q'}^{\Tbf} S = \Delta_{q'}^{\Tbf^{\mkern1mu q}} S
	\end{equation}
	for every finite vertex subset $S\subseteq \Tbf_{ot} \cup\wtilde{A}$. As each $q'$-island 
	$C \subseteq \wtilde{A}$ is finite by Corollary~\ref{VolumeBound} and thus a $q'$-isolated core in $\Tbf$,
	the identity
	\eqref{q-isol-equal} implies that $C$ is also a $q'$-isolated core in $\Tbf^{\mkern1mu q}$ and, hence, 
	$C\subseteq A_{q'}(\Tbf^{\mkern1mu q})$. \\
	``$A_{q'}(\Tbf^{\mkern1mu q}) \subseteq \wtilde{A}$'' \quad
	Let $\emptyset\neq C' \subseteq A_{q'}(\Tbf^{\mkern1mu q})$ be a $q'$-isolated core in $\Tbf^{\mkern1mu q}$.
	In particular, $C'$ is finite. 
	First, we consider the case where $C' \subseteq \Tbf_{ot}^{\mkern1mu q} \cup \wtilde{A}$. In this case, 
	the identity \eqref{q-isol-equal} implies 
	that $C'$ is also a $q'$-isolated core in $\Tbf$, i.e.\ $C' \subseteq A_{q'}(\Tbf)$ and, hence, $C'\subseteq \wtilde{A}$. We now show that the complementary case
	in which there exists a vertex 
	$x \in C' \cap [\Tbf^{\mkern1mu q} \setminus (\Tbf_{ot}^{\mkern1mu q} \cup \wtilde{A})]$ cannot occur.
	Indeed, since $\Tbf^{\mkern1mu q}$ is a tree and $C'$ is finite, it follows that there exists 
	$x' \in C' \cap [\Tbf^{\mkern1mu q} \setminus (\Tbf_{ot}^{\mkern1mu q} \cup \wtilde{A})]$ with 
	$\deg_{\Tbf^{\mkern1mu q}}(x') =z_{t}$ and $\deg_{C'}(x')=1$. By the definition of $C'$ being a 
	$q'$-isolated core of $\Tbf^{\mkern1mu q}$ we have 
	\begin{equation}
		0< \Delta_{q'}^{\Tbf^{\mkern1mu q}}C' - \Delta_{q'}^{\Tbf^{\mkern1mu q}}(C'\setminus\{x'\}) 
		= q' - \big(\deg_{\Tbf^{\mkern1mu q}}(x') - 2\deg_{C'}(x')\big) 
		= q' - z_{t} + 2.
	\end{equation}
	But this is a contradiction, because $q' <1$ and $z_{t} \ge 3$. This finishes the proof of
	\eqref{eq:prime-islands}.  The equality \eqref{eq:prime-islands} implies in particular that $|A_{q'}(\Tbf^{\mkern1mu q})| < \infty$, because 
	Corollary~\ref{VolumeBound} applied to $\Tbf\notin\mathcal N$ guarantees the finiteness of each $q'$-island 
	$C \subseteq A_{q'}(\Tbf)$. Moreover, $|\Tbf|=\infty$ because $\Tbf\notin\mathcal N$ so that $\Tbf \setminus \wtilde{A} \neq \emptyset$, and the construction of $\Tbf^{\mkern1mu q}$ implies that 
	$|\Tbf^{\mkern1mu q}|=\infty$. This finishes the proof of \eqref{prime-islands}.

	Finally, we prove Part~\eqref{prime-prob}. We recall that by the construction of
	$\Tbf^{\mkern1mu q}$ and \eqref{eq:prime-islands}, 
	the tree  $\Tbf_{ot} \cup A_{q}(\Tbf^{\mkern1mu q})$ is an identical subtree of both 
	$\Tbf$ and $\Tbf^{\mkern1mu q}$. 
	Let $x,y \in \Tbf_{ot} \setminus A_{q}(\Tbf) = \Tbf^{\mkern1mu q}_{ot} \setminus A_{q}(\Tbf^{\mkern1mu q})$. 
	In particular, we have 
	\begin{equation}
		\label{prime-prob1}
		w^{(q)}_{\mathrm{srw}}(x) = \deg_{\Tbf^{\mkern1mu q}}(x) = \deg_{\Tbf}(x) = w_{\mathrm{srw}}(x).
	\end{equation}
	Moreover, the simple random walk 
	$\{X_{s}\}_{s\in\Nbb_{0}}$ on $\Tbf$ when restricted to $\Tbf_{ot} \cup A_{q}(\Tbf^{\mkern1mu q}) \subseteq \Tbf$ 
	coincides with the simple random walk $\{X_{s}^{(q)}\}_{s\in\Nbb_{0}}$ on $\Tbf^{\mkern1mu q}$ when 
	restricted to $\Tbf_{ot} \cup A_{q}(\Tbf^{\mkern1mu q}) \subseteq \Tbf^{\mkern1mu q}$. This implies 
	\begin{equation}
		\label{prime-prob2}
		P^{\Tbf^{\mkern1mu q}}_x[X^{(q)}_{\tau^{(q)}_{\Tbf^{\mkern1mu q}_{q}}}=y]
		= P^{\Tbf}_x[X_{\tau_{\Tbf_{q}}}=y],
	\end{equation}
	and the assertion follows from \eqref{prime-prob1} and \eqref{prime-prob2}.
\end{proof}

\begin{proof}[Proof of Theorem~\ref{exactdecay}]
	We fix $t\in\Nbb$, $\Tbf\in (M_{t,z_{t}} \cap H_t^c)\setminus\mathcal{N}$, $q \coloneqq \frac23 h$ and 
	\begin{equation}\label{qt}
		q_{t} \coloneqq \frac{h}{2\sqrt{2} \,(t z_{t})^{\frac13}} 
	\end{equation}
	so that $q_{t} < q$ and, hence, 
	\begin{equation}
		\label{aqt<aq}
		A_{q_t}(\Tbf) \subseteq A_q(\Tbf)
	\end{equation} 
	by Lemma~\ref{Contain}. 
	
	We decompose the return probability of the simple random walk on $\Tbf$ according to 
	\begin{align}
		\label{splitp}
		P_o^\Tbf[X_t=o] 
		&= P^{\Tbf}_o\big[X_t=o \wedge \forall s\in\{1,...,t\}:\; X_s\in 
				\Tbf_{ot} \setminus A_{q_{t}}(\Tbf) \big]
			 \notag \\
		&\quad + P^{\Tbf}_o\big[X_t=o \wedge \exists\, s\in\{1,...,t\}:\; X_s\in \Tbf_{ot} \cap 
				A_{q_{t}}(\Tbf)\big] \notag \\
		&= P^{\Tbf^{\mkern1mu q_{t}}}_o \big[X^{(q_{t})}_t=o \wedge \forall s\in\{1,...,t\}:\; X_s^{(q_{t})}\in 
				\Tbf^{\mkern1mu q_{t}}_{ot} \setminus A_{q_{t}}(\Tbf^{\mkern1mu q_{t}}) \big]
			 \notag \\
		&\quad + P^{\Tbf^{\mkern1mu q}}_o \big[X_t^{(q)}=o \wedge \exists\, s\in\{1,...,t\}:\; 
				X^{(q)}_s\in \Tbf^{\mkern1mu q}_{ot} \cap A_{q_{t}}(\Tbf^{\mkern1mu q})\big]. 
	\end{align}
	As for the second equality, we note that the regularised trees satisfy 
	$\Tbf_{ot} =\Tbf_{ot}^{\mkern1mu q'}$ and
	$\Tbf_{ot} \cap A_{q_{t}}(\Tbf) = \Tbf_{ot}^{\mkern1mu q'} \cap A_{q_{t}}(\Tbf^{\mkern1mu q'})$,
	which follows from Lemma~\ref{prime}\eqref{prime-islands}, for both $q'=q_{t}$ and $q'=q$.
		
	Next, we estimate the probability in the third line of \eqref{splitp}.
	The fact that $\Tbf^{\mkern1mu q_{t}}$ is a tree implies that the simple random walk in this
	probability jumps only between vertices in 	
	$\Tbf^{\mkern1mu q_{t}}_{q_{t}}$ 
	and such that no two consecutive vertices $x,y$ in any path 
	belong to the outer vertex boundary of the same $q_{t}$-island of $\Tbf^{\mkern1mu q_{t}}$. 
	This means that for each jump, we have the equality 
	$w^{(q_{t})}_{\text{srw}}(x,y)=w^{(q_{t})}_{q_{t}}(x,y)$, cf.\ \eqref{wqedge}. Therefore, the  
	estimate
	\begin{equation}
		P^{\Tbf^{\mkern1mu q_{t}}}_o \big[X_t^{(q_{t})}=o \wedge \forall s\in\{1,...,t\}:\; 
				X_s^{(q_{t})}\in \Tbf^{\mkern1mu q_{t}}_{ot} \setminus 
				A_{q_{t}}(\Tbf^{\mkern1mu q_{t}}) \big]
		\le P^{\Tbf^{\mkern1mu q_{t}}_{q_{t}}}_{o} [W_{t}^{(q_{t})}=o]
	\end{equation}
	holds, where the inequality arises because the requirements that 
	$\{W_{s}^{(q_{t})}\}_{s\in\{1,\ldots,t\}}$ must not jump over $q_{t}$-islands or is forbidden 
	to stay at a vertex have been dropped. 
	Rewriting the right-hand side in terms of the associated Markov operator 
	$\mathbf{P}_{\Tbf^{\mkern1mu q_{t}}_{q_t}}$ on 
	the weighted Hilbert space $\ell^{2}(\Tbf^{\mkern1mu q_{t}}_{q_t},w^{(q_{t})}_{q_{t}})$, we obtain
	\begin{align}
		\label{estp'}
		P^{\Tbf^{\mkern1mu q_{t}}}_o & \big[X_t^{(q_{t})}=o \wedge \forall s\in\{1,...,t\}:
				\; X_s^{(q_{t})}\in \Tbf^{\mkern1mu q_{t}}_{ot} \setminus A_{q_{t}}(\Tbf^{\mkern1mu q_{t}}) \big] \notag \\
		&\le \frac{1}{w^{(q_{t})}_{q_{t}}(o)} \, \langle 1_{\{o\}} ,
			 \mathbf{P}_{\Tbf^{\mkern1mu q_{t}}_{q_{t}}, w^{(q_{t})}_{q_{t}}}^{\mkern2mu t} 
					1_{\{o\}} \rangle_{\Tbf^{\mkern1mu q_{t}}_{q_t}, \,	w^{(q_{t})}_{q_{t}}} 
			\leq \|\mathbf{P}_{\Tbf^{\mkern1mu q_{t}}_{q_{t}}, w^{(q_{t})}_{q_{t}}} \|_{\Tbf^{\mkern1mu q_{t}}_{q_{t}},\,
					 w^{(q_{t})}_{q_{t}}}^t 
			\leq  \Bigl( 1-\frac{q_{t}^{2}}{9} \Bigr)^{\frac{t}{2}} \notag \\ 
		&	\leq \exp\bigg[- \frac{q_{t}^{2} t}{18}\bigg]
			= \exp\bigg[- \frac{h^2}{144}\,\Big(\frac{t}{z_{t}^{2}}\Big)^{\frac13}\bigg], 
	\end{align}
	where the last inequality in the second line follows from an application of 
	Theorem~\ref{MarkovKernel} and Lemma~\ref{wic} to the weighted graph 
	$(\Tbf^{\mkern1mu q_{t}}_{q_{t}}, w^{(q_{t})}_{q_{t}})$. 
	This is justified because 
	$|\Tbf^{\mkern1mu q_{t}}_{q_{t}}|=\infty$, see Lemma~\ref{prime}\eqref{prime-islands}. 
	The inequality in the last line follows from $\ln(1+u)\leq u$, $|u|<1$, which is applicable
	by the definitions of $q_{t}$ and $z_{t}$ and due to $h<1$.
	
	Before we estimate the probability in the last line of \eqref{splitp}, we need to introduce one more 
	notion. Let 
	\begin{equation}
		A_{q,t} \coloneqq \bigcup_{q\text{-islands } C \subseteq A_q(\Tbf^{\mkern1mu q}) \text{ with }|C|> 
			\frac{1}{q_{t}} } C	\subseteq A_q(\Tbf^{\mkern1mu q})
	\end{equation}
	be the union of all $q$-islands $C$ in $\Tbf^{\mkern1mu q}$ with volume $|C|> \frac{1}{q_{t}}$. 
	We remark that by construction of $\Tbf^{\mkern1mu q}$, all such $q$-islands $C$ obey 
	$\dist(o,C) \le t$.
	Applying Lemma~\ref{Contained} with $q'=q_{t}$ to any of the remaining $q$-islands 
	$S \subseteq A_q(\Tbf^{\mkern1mu q})\setminus A_{q,t}$, 
	gives $S\subseteq \Tbf^{\mkern1mu q}\setminus A_{q_{t}}(\Tbf^{\mkern1mu q})$ so that 
	\eqref{aqt<aq} with $\Tbf$ replaced by $\Tbf^{\mkern1mu q}$ can be sharpened to 
	\begin{equation}
		\label{qat}
		A_{q_{t}}(\Tbf^{\mkern1mu q})	\subseteq A_{q,t}.
	\end{equation}
	Thus, the probability in the last line of \eqref{splitp} can be estimated as 
	\begin{equation}
		\label{second-prop}
		P^{\Tbf^{\mkern1mu q}}_o\big[X_t^{(q)}=o \wedge \exists\, s\in\{1,...,t\}:\; 
			X_s^{(q)}\in  \Tbf^{\mkern1mu q}_{ot} \cap A_{q_{t}}(\Tbf^{\mkern1mu q}) \big]		 
		\le P^{\Tbf^{\mkern1mu q}}_o\big[\exists\, s\in\{1,...,t\} :\; X_s^{(q)}\in A_{q,t}].
	\end{equation}
	In order to proceed further, we define the $q$-\emph{territory} of a
	$q$-island $C\subseteq A_{q,t}$ by 
	\begin{equation}\label{territory}
		D_C \coloneqq \Bigl\{x\in\Tbf^{\mkern1mu q} :\;\dist_q(x,C) \leq \frac{q}{4q_{t}z_t} \Bigr\}
	\end{equation}
	and assert two claims.

	\noindent
	\emph{Claim 1.} \quad $o\notin D_{C}$ for any $q$-island $C\subseteq A_{q,t}$.

	\noindent
	In view of Lemma~\ref{prime}\eqref{prime-islands}, Claim~1 will be obtained from the following argument: 
	We assume that $o\in D_C$ for some $q$-island 
	$C\subseteq A_{q}(\Tbf) $ with $\dist(o,C)\leq t$ and $|C| > \frac{1}{q_{t}}$ and strive for a 
	contradiction. 
  In fact, given these assumptions we conclude $\Tbf\in H_{q,t}^{0}$ because $C$ qualifies as $U_{q,t}$ in the definition \eqref{defHt}. 
  Indeed, since $z_{t} \ge 2$ we have $q|C| > \frac{q}{q_{t}} \ge 2^{\frac56} t^{\frac13}$.
  Furthermore, since $\Tbf \notin\mathcal{N}$ we have $|C| <\infty$ by Corollary~\ref{VolumeBound} and, 
  finally, since $\dist(o,C)\leq t$, there exists a bridge $B_{q,t}$ connecting the root $o$ with $C$ 
  and satisfying $\max_{v\in B_{q,t}}\dist(o,v)\leq t$. 
  Without loss of generality we assume that $B_{q,t}$ is the bridge with the shortest $q$-length 
  among all such bridges. Then, 
	\begin{equation}
		z_t\frac{|B_{q,t}\cup\{o\}\setminus A_q|}{|C|}
		= z_t\frac{\dist_q(o,C)}{|C|}
		\leq z_t\frac{qq_t}{4q_tz_t}
		<\frac{h}{3}
	\end{equation}
	where we used $o\in D_C$ and $|C| > q_{t}^{-1}$ for the first inequality. 
	It follows that even $\Tbf\in H_{t}$ according to the definition \eqref{territoryineq}. 
	This contradicts the initial assumption $\Tbf\in H_t^c$ and completes the proof of Claim~1.

	\noindent
	\emph{Claim 2.} \quad There exist at most $t$-many (distinct) $q$-islands $C_j\subseteq A_{q,t}$, 
	$j\in\{1,\ldots,t\}$, such that their territories 
	form a connected set $\bigcup_{j=1}^{t}D_{C_j}$ of vertices.

	\noindent
	We prove Claim~2 by contradiction and assume, again in view of Lemma~\ref{prime}\eqref{prime-islands}, 
	that there exists a union 
	$U_{q,t} \coloneqq\bigcup_{j=0}^{t}C_j$ of $(t+1)$-many $q$-islands in $\Tbf$ with 
	$\dist(o,C)\leq t$ and $|C| > \frac{1}{q_{t}}$. 
	Then, there is a bridge structure $B_{q,t}$ interconnecting $U_{q,t}\cup\{o\}$ with
	$q|U_{q,t}| > (t+1)  \frac{q}{q_{t}} > \frac{q}{q_{t}} \ge 2^{\frac56} t^{\frac13}$, 
	$\max_{v\in B_{q,t}}\dist(o,v)\leq t$, for which we used that $\Tbf$ is a tree, and 
	\begin{equation}
		\label{bridge-size}
		|(B_{q,t}\cup\{o\})\setminus A_q| \leq t+t \Big(\frac{q}{2 q_{t}z_t}-1\Big).
	\end{equation}
	The bound \eqref{bridge-size} holds because the bridge structure $B_{q,t}$ requires 
	at most $t$ vertices to connect the root $o$ with one of the islands, and in order to connect 
	this island with the remaining $t$ islands it requires $t$ further bridges 
	$B_j$, $j\in\{1,...,t\}$, between two islands. Each such bridge $B_j$ will be chosen to pass through a 
	common vertex $v_j$ of the territories of the two islands $C^{(j)}_1$ and $C^{(j)}_2$ 
	which it connects so that 
	\begin{equation}\label{bk}
		|B_{j}\setminus A_q| \leq\dist_q(v_j,C^{(j)}_1) + \dist_q(v_j,C^{(j)}_2)-1
		\le \frac{q}{2 q_{t} z_t}-1.
	\end{equation}
	Here we used \eqref{territory} for the second bound. This justifies \eqref{bridge-size}.
	Finally, we infer from \eqref{bridge-size} that 
	\begin{equation}
		\label{bridge-size2}
		z_{t} \,\frac{|(B_{q,t}\cup\{o\})\setminus A_q|}{|U_{q,t}|} 
		\leq t \frac{q}{2 q_{t} |U_{q,t}|}
		\leq \frac{tq}{2(t+1)} \le \frac{h}{3}.
	\end{equation}
	It follows from \eqref{territoryineq} that $\Tbf \in H_{t}$ which contradicts the initial assumption 
	$\Tbf\in H_t^c$. The proof of Claim~2 is complete. 

	Now, there are finitely many ``groups'' 
	\begin{equation}
		\mathcal{C}_{r} \coloneqq \bigcup_{j=1}^{J_{r}}C_{j}^{(r)},
	\end{equation} 
	of $q$-islands in $\Tbf^{\mkern1mu q}$, where $r\in \{1,\ldots, R\}$ for some $R\in\Nbb$, 
	each group---according to Claim~2---consisting of at most $t$-many $q$-islands 
	$C_{j}^{(r)} \subseteq A_{q,t}$ with $j\in\{1,\ldots,J_{r}\}$, $J_{r} \in \{1,\ldots,t\}$, and
	such that the territories of $q$-islands from different groups are disjoint. 
	Moreover, for every $r\in\{1,\ldots, R\}$, the union of territories 
	$D_{\mathcal{C}_{r}} \coloneqq \bigcup_{j=1}^{J_{r}} D_{C_{j}^{(r)}}$ 
	within each group is connected and possesses a unique vertex 
	$y_{r} \in D_{\mathcal{C}_{r}}$ which is closest to the root because $\Tbf^{\mkern1mu q}$ 
	is a tree. It follows from Claim~1 that 
	$y_{r}$ belongs to the inner vertex boundary of $D_{\mathcal{C}_{r}}$ and therefore 
	\begin{equation}
		\label{doc}
		\dist_q(y_{r},\mathcal{C}_{r})=\lfloor\frac{q}{4q_{t}z_t}\rfloor.
	\end{equation}
	Hence, the probability on the right-hand side of \eqref{second-prop} can be estimated as
	\begin{allowdisplaybreaks}
	\begin{align}
		\label{second-prob2}
		P^{\Tbf^{\mkern1mu q}}_o\big[\exists\, & s\in\{1,...,t\} :\; X_s^{(q)}\in A_{q,t}\big] \notag \\
		&\le P^{\Tbf^{\mkern1mu q}}_o\big[ \exists\, r \in\{1,\ldots,R\} \;
				\exists\, s_{0}\in\{1,...,t\} \; \exists\, s \in\{s_{0}+1,\ldots,t\} \notag \\
		&\hspace{6cm}	:\; X_{s_{0}}^{(q)} = y_{r} \text{ and } X_{s}^{(q)} \in\mathcal{C}_{r} \big] \notag \\
		&\le \sum_{s_{0}=1}^{t} \sum_{r=1}^{R} P^{\Tbf^{\mkern1mu q}}_o\big[ 
				\exists\, s \in\Nbb\setminus\{1,\ldots,s_{0}\} :\; X_{s_{0}}^{(q)} = y_{r} \text{ and } 
				X_{s}^{(q)} \in\mathcal{C}_{r} \big] \notag \\
		&= \sum_{s_{0}=1}^{t} \sum_{r=1}^{R} E^{\Tbf^{\mkern1mu q}}_o\Big[ 1_{\{y_{r}\}}(X_{s_{0}}^{(q)})
				P^{\Tbf^{\mkern1mu q}}_{y_{r}}[	\tau_{\mathcal{C}_{r}} < \infty ] \Big],  
	\end{align}
	\end{allowdisplaybreaks}%
	where the equality rests on the Markov property and $E^{\Tbf^{\mkern1mu q}}_o$ is the 
	probabilistic expectation corresponding to $P^{\Tbf^{\mkern1mu q}}_o$.
	Abbreviating $\mathcal{S}_{q} \coloneqq \sup_{r\in\{1,\ldots,R\}} P^{\Tbf^{\mkern1mu q}}_{y_{r}}
		[\tau_{\mathcal{C}_{r}} < \infty]$ and noting that the $y_{r}$'s are pairwise distinct, 
	we conclude from \eqref{second-prob2} 
	\begin{equation}
		\label{second-prob3}
		P^{\Tbf^{\mkern1mu q}}_o\big[\exists\,  s\in\{1,...,t\} :\; X_s^{(q)}\in A_{q,t}\big] 
		\le \mathcal{S}_{q} \sum_{s_{0}=1}^{t} P^{\Tbf^{\mkern1mu q}}_o \Big[ X_{s_{0}}^{(q)} \in 
				\bigcup_{r=1}^{R}\{y_{r}\} \Big]	\le t	\mathcal{S}_{q}.
	\end{equation}
	The supremum $\mathcal{S}_{q}$ can be estimated with Lemma~\ref{TimeToEscape}, choosing $\Tbf$
	there as the regularised tree $\Tbf^{\mkern1mu q}$. This is possible because of 
	Lemma~\ref{prime}\eqref{prime-bdd} and~\eqref{prime-islands} and gives
	\begin{equation}
		\label{second-prob4}
		\mathcal{S}_{q} \le \frac{18}{q^2} \,\Bigl(1-\frac{q^2}{9}\Bigr)^{\frac{q}{8q_{t}z_{t}} - \frac32}\;
			\,(tz_{t})^{\frac12}.
	\end{equation}
	Combining \eqref{second-prob3} and \eqref{second-prob4}, we infer that there exists 
	$t_{0}\in\Nbb$, which depends only on $h$ and on the sequence $(z_t)_{t\in\Nbb}$, such that 
	\begin{equation}
		\label{second-prob5}
		P^{\Tbf^{\mkern1mu q}}_o\big[\exists\, s\in\{1,...,t\} :\; X_s^{(q)}\in A_{q,t}\big] 
		\le \exp\bigg[- \frac{h^2}{86}\,\Big(\frac{t}{z_{t}^{2}}\Big)^{\frac13}\bigg]
	\end{equation}
	holds, provided $t\geq t_{0}$. Thus, the theorem follows from \eqref{splitp}, \eqref{estp'},  
	\eqref{second-prop} and \eqref{second-prob5}.
\end{proof}

\noindent
Finally, we will prove Theorems~\ref{Decay} and \ref{thm}.	

\begin{proof}[Proof of Theorem~\ref{Decay}]
	The offspring distribution has bounded support by hypothesis of the theorem. 
	Thus, there is $z\in\Nbb$ such that 
	$p_j=0$ for all $j\geq z$, and we choose $z_t \coloneqq  \max\{3, z\}$ for every $t\in\Nbb$. 
	Moreover, $\Tbb \setminus M_{t,z_{t}}$ is a $G$-null set for every $t\in\Nbb$ so that
	Theorem~\ref{exactdecay} and Lemma~\ref{badlargeislands} imply
	\begin{align}
		GP[X_t=o] 
		&\le \int_{M_{t,z_{t}}\cap H_{t}^{c}}\d G(\Tbf) P_{o}^{\Tbf}[X_t=o] + G[M_{t,z_{t}} \cap H_{t}]  
			\notag \\
		&\le \exp\bigg[- \frac{h^2}{144 z^{\frac23}} \, t^{\frac13} \bigg]
			+ \exp \big[ -c_{5} t^{\frac13}\big] 
	\end{align}
	for all $t\geq t_0$, where $t_{0}$ depends only on $h$ and $z$, and the constant $c_{5}>0$ is 
	defined in Lemma~\ref{lem:hdef}.
	By the same argument as in Remark~\ref{PiauRem}\eqref{noConstant} we obtain the claim of the theorem. 
\end{proof}

\begin{proof}[Proof of Theorem~\ref{thm}]
	We consider an offspring distribution with super-Gaussian decay as in \eqref{eq:k-decay} for some constants
	$c_{1}, c_{2} >0$ and $k > 2$. Let $z_{t} \coloneqq 3+ c_{3}t^{\frac1k}$ for every $t\in\Nbb$ 
	with $c_{3}> 0$ as required by Lemma~\ref{SetsOfTrees}. In particular, we then have 
	$z_t\in\;\scriptstyle\mathcal{O}$$(t^{\frac12})$ as $t\to\infty$ and $F_t^c\subseteq M_{t,z_{t}}$, where 
	the former is defined in \eqref{Ftdefeq} and the latter in \eqref{Mt-def}. We conclude
	\begin{align}
		GP[X_t=o] 
		&\le \int_{F_{t}^{c}\cap H_{t}^{c}}\d G(\Tbf) P_{o}^{\Tbf}[X_t=o] 
				+ G[F_{t}] + G[M_{t,z_{t}} \cap H_{t}] \notag \\
		&\le \exp\bigg[- \frac{h^2}{144 (3+c_{3})^{\frac23}} \, t^{\frac13 -\frac{2}{3k}} \bigg]
			+ C \exp \big[ -c_{4} t\big] + \exp \big[ -c_{5} t^{\frac13}\big], 
	\end{align}
	where the second inequality follows from Theorem~\ref{exactdecay},
	Lemma~\ref{SetsOfTrees} and Lemma~\ref{badlargeislands} and holds for all $t\ge t_{0}$ which 
	arises from Theorem~\ref{exactdecay}. By the same argument as in Remark~\ref{PiauRem}\eqref{noConstant} we infer 
	the claim of the theorem.
\end{proof}

\begin{appendix}
\section{On \texorpdfstring{\boldmath$q$}{q}-islands and \texorpdfstring{\boldmath$q$}{q}-oceans}

For completeness and convenience of the reader, we recall in this appendix some basic properties of 
$q$-islands and $q$-oceans, which are taken from \cite[Sect.~3]{virag2000bnddegree}. 
Some slight modifications occur because our notion of volume refers to the cardinality of a set, whereas 
Vir\'ag considers weighted volumes. Throughout, $\Tbf$ can be any fixed infinite and connected graph with locally bounded 
vertex degrees. It need not be a realisation of a Galton--Watson tree here.   

\begin{lemma}\label{DeltaBound}
	Let $q>0$, let $A\subseteq\Tbf$ be a finite vertex subset and let $S \subseteq\Tbf$ be a $q$-isolated 
	core. Then, we have $\Delta_q A \le \Delta_q(A\cup S)$ with equality if and only if $S\subseteq A$. 
\end{lemma}

\begin{proof}
	If $S\subseteq A$, then the claim trivially holds with equality. 
	So let us now suppose that $S$ is not a subset of $A$.
	
	We note that if $B$ and $C$ are finite disjoint vertex subsets of $\Tbf$, then
	\begin{equation}\label{DeltaDisjSplit}
	\Delta_q(B\cup C)=\Delta_q B+\Delta_q C+ 2 |\partial B\cap\partial C|. 
	\end{equation}
	The factor $2$ in the above expression appears since common boundary edges of $B$ and $C$ are not 
	boundary edges of their union, i.e., $2|\partial B\cap\partial C| = |\partial B| + |\partial C| -
	|\partial B\cup\partial C|$.

	We conclude from \eqref{DeltaDisjSplit} that $\Delta_q(A\cup S)=\Delta_q (A\setminus S)+\Delta_q S+ 2\vv{\partial (A\setminus S)\cap\partial S}$. Since we assumed that $S$ is a $q$-isolated core, we have $\Delta_q S> \Delta_q (A\cap S)$ by definition because $A \cap S \subsetneq S$ due to $S$ not being a subset of $A$. Also, $\partial (A\setminus S)\cap\partial S\supseteq \partial (A\setminus S)\cap\partial (A\cap S)$, since every edge in the intersection of sets has to connect $S$ with its complement and is thus in $\partial S$. Therefore, another application of \eqref{DeltaDisjSplit} yields 
	\begin{equation}
	\Delta_q(A\cup S)>\Delta_q (A\setminus S)+\Delta_q (A\cap S)+2\vv{\partial (A\setminus S)\cap\partial (A\cap S)}=\Delta_q A. 
	\end{equation} 
\end{proof}

\begin{corollary}\label{finite}
	Let $q>0$. Then, the union of finitely many $q$-isolated cores of $\Tbf$ is a $q$-isolated 
	core of $\Tbf$.
\end{corollary}

\begin{proof}
	It suffices to prove the claim for two $q$-isolated cores $S$ and $S'$ of $\Tbf$. Let $A\subsetneq S\cup S'$ be arbitrary. 
	Then $A$ must be a proper subset of least one of the sets $S$ and $S'$. 
	W.l.o.g.\ suppose that $A$ is a proper subset of $S$. Applying Lemma~\ref{DeltaBound} with $A$ and $S$, followed by another application with $A\cup S$ and $A \cup S \cup S'$, yields
	$\Delta_q A < \Delta_q (A\cup S) \le \Delta_q (A\cup S\cup S')=\Delta_q (S\cup S')$. The last equality holds because of $A\subset S\cup S'$, and the claim follows.
\end{proof}

\noindent 
The following lemma relates to a statement in \cite[Sect.~3]{virag2000bnddegree} 
which is given there without proof.

\begin{lemma}\label{det-VolumeBound}
	Let $q \in \,]0, \mathbf{i}(\mathbf{T})[\,$. Then, every 
	$q$-island of $\Tbf$ has only finitely many vertices and thus is itself a $q$-isolated core of $\Tbf$. 
\end{lemma}

\begin{proof}
	Suppose that there exists a $q$-island $S\subseteq\Tbf$ with 
	$|S| = \infty$. Thus, $S$ must be formed by a countably infinite union $S=\bigcup_{j\in\Nbb}S_{j}$
	of $q$-isolated cores $S_{j}$ of $\Tbf$. Then, $A_{n}\coloneqq\bigcup_{j=1}^{n}S_{j}$ is a $q$-isolated core 
	for every $n\in\Nbb$ by Corollary~\ref{finite}. Hence, we have 
	\begin{equation}
		\label{An-iso}
		\frac{|\partial A_{n}|}{|A_{n}|} < q 
	\end{equation}
	for every $n\in\Nbb$ by Remark~\ref{core-rem}\eqref{core-isolated}.	
	W.l.o.g.\ it can be assumed that each 
	$S_{j}$ is not empty and, due to Remark~\ref{core-rem}\eqref{core-connected}, connected. 
	Since $S$ is connected 
	by hypothesis a suitable renumbering of the $S_{j}$'s will guarantee that $A_{n}$ is connected 
	for every $n\in\Nbb$. Furthermore, we can assume w.l.o.g.\ that 
	$S_{j+1}\setminus A_{j} \neq\emptyset$ for every $j\in\Nbb$. Thus, $|A_{n}| \ge n$ for every
	$n\in\Nbb$. Finally, we connect $A_{n}$ with the root $o$ for every $n\in\Nbb$ by attaching a 
	suitable linear path $P_{n} \subset\Tbf$ to it. If $o\in A_{n}$ already, we set $P_{n}= \emptyset$. 
	Since $A_{n} \subseteq A_{n+1}$, we have $P_{n} \supseteq P_{n+1}$, and because of the linear 
	structure of 
	$P_{n}$, this implies $|\partial P_{n}| \ge |\partial P_{n+1}|$ for every $n\in\Nbb$. 
	Defining $K_{n} \coloneqq P_{n} \cup A_{n}$ for $n\in\Nbb$, we conclude that $o \in K_{n} \subseteq\Tbf$ is 
	connected, $|K_{n}| \ge |A_{n}| \ge n$ and $|\partial K_{n}| \le |\partial A_{n}| + |\partial P_{n}|
		\le |\partial A_{n}| + |\partial P_{1}|$ for every $n\in\Nbb$.
	We thus infer a contradiction in that 
	\begin{equation}
		\label{inf-island-cons}
		\mathbf{i}(\Tbf) \le \lim_{n\to\infty} \frac{|\partial K_{n}|}{|K_{n}|} \le q,
	\end{equation}
	where we used \eqref{An-iso} for the last estimate. 
	Hence, every $q$-island of $\Tbf$ is finite, therefore a finite union of $q$-isolated cores and 
	therefore itself a $q$-isolated core by Corollary~\ref{finite}.
\end{proof}

\noindent 
Next, we argue that decreasing $q$ raises the sea level of the oceans.

\begin{lemma}\label{Contain}
	Let $0<q'<q$. Then, $A_{q'}\subseteq A_q$.
\end{lemma}

\begin{proof}
	We have $\Delta_q S=(q-q')\vv{S}+\Delta_{q'} S\geq\Delta_{q'} S$ for any finite vertex subset $S\subseteq\Tbf$. So any $q'$-isolated set is also $q$-isolated. Moreover, if $A\subsetneq S$ with $\Delta_{q'} A<\Delta_{q'} S$, then also $\Delta_{q} A<\Delta_{q} S$. Therefore, $q'$-isolated cores are $q$-isolated cores as well, giving $A_{q'}\subseteq A_q$.
\end{proof}

\noindent
In the next lemma we quantify the preceding statement in that too small $q$-islands sink into the oceans 
when lowering $q$.

\begin{lemma}\label{Contained}
	Let $0<q'<q$ and $S\subseteq \Tbf$ be a union of $q$-islands with $\vv{S} \leq\frac{1}{q'}$. Then, $S\subseteq \Tbf\setminus A_{q'}$.
\end{lemma}

\begin{proof}
	We argue by contradiction and assume that there exists $\emptyset\neq S'\subseteq S$ with $S'\subseteq A_{q'}$. Since $S$ is
	a finite union of $q$-islands and $A_{q'} \subseteq A_{q}$ by Lemma~\ref{Contain}, it follows that $S'$ is a finite union of 
	$q'$-islands and, thus, a $q'$-isolated core, i.e.\ $\Delta_{q'}S' >0$. On the other hand, 
	\begin{equation}
		\Delta_{q'}S' \le q' |S| - |\partial S'| \le 1 - |\partial S'| \le 0,
	\end{equation}
	where we used the volume assumption for $S$ in the second inequality and $|\partial S'| \ge 1$ in the last inequality.
	This holds because $\Tbf$ is infinite and connected.
\end{proof}

\end{appendix}



\begin{thebibliography}{10}

\bibitem{virag2017spectrans}
Bordenave, C., Sen, A. and Vir\'ag, B. {{Mean quantum percolation}}. \emph{J.
  Eur. Math. Soc. (JEMS)} \textbf{16} (2017), 3679--3707. \MR{3730511}

\bibitem{ChenPeres2004}
Chen, D. and Peres, Y. {{Anchored expansion, percolation and speed.
  With an appendix by G\'abor Pete}}. \emph{Ann. Probab.} \textbf{32} (2004), 
  2978--2995. \MR{2094436}

\bibitem{Collevecchio.2006}
Collevecchio, A. {{On the transience of processes defined on
  Galton–Watson trees}}. \emph{Ann. Probab.} \textbf{34} (2006), 870--878. \MR{2243872}
  
\bibitem{nina2001speed}
Dembo, A., Gantert, N., Peres, Y. and Zeitouni, O. {{Large deviations for
  random walks on Galton–Watson trees: averaging and uncertainty}}. \emph{Probab.
  Theory Relat. Fields} \textbf{122} (2001), 241--288. \MR{1894069}

\bibitem{DiaconisStroock.1991.GeometricBounds}
Diaconis, P. and Stroock, D. {{Geometric bounds for eigenvalues of
  Markov chains}}. \emph{Ann. Probab.} \textbf{1} (1991), 36--61. \MR{1097463}

\bibitem{GRIMMETT.1984}
Grimmett, G. and Kesten, H. {{Random electrical networks on
  complete graphs}}. \emph{J. London Math. Soc.} \textbf{30} (1984), 171--192. \MR{0760886}

\bibitem{Grimmett.2001}
Grimmett, G. and Kesten, H. (2001) {Random electrical networks on complete graphs {II}: {P}roofs}.
  \href{https://doi.org/10.48550/arxiv.math/0107068}{arXiv:math/0107068} 

\bibitem{peres2016prob}
Lyons, R. and Peres, Y. \emph{Probability on trees and networks}. Cambridge
  University Press, Cambridge, 2016. \MR{3616205}

\bibitem{piau1998lowbnd}
Piau, D. {{Th\'eor{e}me central limite fonctionnel pour une marche au
  hasard en environnement al\'eatoire}}. \emph{Ann. Probab.} \textbf{26} (1998),
  1016--1040. \MR{1634413}

\bibitem{hofstad2020glimit}
v.~d.\ Hofstad, R. \emph{Random graphs and complex networks}, vol.~2. Cambridge
  University Press, Cambridge, 2024. 
  \href{https://doi.org/10.1017/9781316795552}{DOI:10.1017/9781316795552}

\bibitem{virag2000bnddegree}
Vir\'ag, B. {{Anchored expansion and random walk}}. \emph{Geom. Funct. Anal.}
  \textbf{10} (2000), 1588--1605. \MR{1810755}

\bibitem{MR1743100}
Woess, W. \emph{Random walks on infinite graphs and groups}. Cambridge
  University Press, Cambridge, 2000. \MR{1743100}

\end{thebibliography}




\end{document}